\documentclass{article} 

\pdfoutput=1  %  for arxiv

\usepackage{fncylab,enumerate,wrapfig,placeins}

\usepackage{amsmath}

\usepackage{amsthm,amsfonts,amssymb, graphicx}  
\usepackage{relsize} % mathlarger
\usepackage{accents}  %just four doublehat 

\usepackage{caption}
 \usepackage{textgreek,upgreek,bm}

 \usepackage{titlesec}

\usepackage{geometry}
\def\urls#1{{\footnotesize\url{#1}}}

\def\mindex#1{\index{#1}}

 \def\oc{\star}   % for optimal control
\def\ocp{*}   % for optimal parameter or function approximation

\DeclareFontFamily{U}{mathx}{\hyphenchar\font45}
\DeclareFontShape{U}{mathx}{m}{n}{<-> mathx10}{}
\DeclareSymbolFont{mathx}{U}{mathx}{m}{n}
\DeclareMathAccent{\widebar}{0}{mathx}{"73}

\def\barUpupsilon{\widebar{\Upupsilon}}

\def\tilnabla{{\widetilde{\nabla}\!}}
\newcommand{\qsaprobe}{{\scalebox{1.1}{$\upxi$}}}  %\upzeta \textphi
\newcommand{\bfqsaprobe}{{\scalebox{1.1}{$\bm{\upxi}$}}}

\def\SAtime{\uptau}

\def\Lip{\ell}
\def\Obj{L}

\def\preODEstate{\Uppsi} %\Upxi BAD %%    \Upupsilon   
\def\bfpreODEstate{\bm{\preODEstate}}

\def\ODEstate{\Uptheta} %\xi
\def\bfODEstate{\bm{\Uptheta}}
\def\barODEstate{\widebar{\Uptheta}}
\def\tilODEstate{\widetilde{\Uptheta}}
\def\haODEstate{\widehat{\Uptheta}}
\def\bfhaODEstate{\widehat{\bfODEstate}}

\def\odestate{\upvartheta}
\def\bfodestate{\bm{\upvartheta}}

\def\fee{\upphi}
\def\feex{\widetilde{\fee}}

\def\Hor{\mathcal{T}} %H}
\def\elig{\zeta}

\def\uH{\underline{H}}
\def\uQ{\underline{Q}}

\newcommand{\bbblot}{\raise1pt\hbox{\vrule height .4ex width .4ex depth .05ex}}
\long\def\defbox#1{\framebox[.9\hsize][c]{\parbox{.85\hsize}{%
\parindent=0pt
\baselineskip=12pt plus .1pt      % STYLE 
\parskip=6pt plus 1.5pt minus 1pt % CHANGES
 #1}}}

\long\def\beginbox#1\endbox{\subsection*{}%
\hbox{\hspace{.05\hsize}\defbox{\medskip#1\bigskip}}%
\subsection*{}}

\def\endbox{}

 \def\archival#1{} %Notes I'd like to save

\def\llbracket{[\![}
\def\rrbracket{]\!]}

\def\FRAC#1#2#3{\genfrac{}{}{}{#1}{#2}{#3}}

\def\ddt{{\mathchoice{\FRAC{1}{d}{dt}}%
{\FRAC{1}{d}{dt}}%
{\FRAC{3}{d}{dt}}%
{\FRAC{3}{d}{dt}}}}

\def\ddw{{\mathchoice{\FRAC{1}{d}{dw}}%
{\FRAC{1}{d}{dw}}%
{\FRAC{3}{d}{dw}}%
{\FRAC{3}{d}{dw}}}}

\def\ddtp{{\mathchoice{\FRAC{1}{d^{\hbox to 2pt{\rm\tiny +\hss}}}{dt}}%
{\FRAC{1}{d^{\hbox to 2pt{\rm\tiny +\hss}}}{dt}}%
{\FRAC{3}{d^{\hbox to 2pt{\rm\tiny +\hss}}}{dt}}%
{\FRAC{3}{d^{\hbox to 2pt{\rm\tiny +\hss}}}{dt}}}}

\def\ddyp{{\mathchoice{\FRAC{1}{d^{\hbox to 2pt{\rm\tiny +\hss}}}{dy}}%
{\FRAC{1}{d^{\hbox to 2pt{\rm\tiny +\hss}}}{dy}}%
{\FRAC{3}{d^{\hbox to 2pt{\rm\tiny +\hss}}}{dy}}%
{\FRAC{3}{d^{\hbox to 2pt{\rm\tiny +\hss}}}{dy}}}}

\def\half{{\mathchoice{\FRAC{1}{1}{2}}%
{\FRAC{1}{1}{2}}%
{\FRAC{3}{1}{2}}%
{\FRAC{3}{1}{2}}}}

\def\fourth{{\mathchoice{\FRAC{1}{1}{4}}%
{\FRAC{1}{1}{4}}%
{\FRAC{3}{1}{4}}%
{\FRAC{3}{1}{4}}}}

\def\limsup{\mathop{\rm lim{\,}sup}}

\def\argmin{\mathop{\rm arg{\,}min}}

\def\Dcs{\text{\rm f}}

\def\state{{\sf X}}

\def\ustate{{\sf U}}

\def\bfmath#1{{\mathchoice{\mbox{\boldmath$#1$}}%
{\mbox{\boldmath$#1$}}%
{\mbox{\boldmath$\scriptstyle#1$}}%
{\mbox{\boldmath$\scriptscriptstyle#1$}}}}

\def\bfPhi{\bfmath{\Phi}}

\def\bfma{\bfmath{a}}

\def\bfmu{\bfmath{u}} 
 
\def\bfmx{\bfmath{x}}

\def\bfmz{\bfmath{z}}

\def\bfmY{\bfmath{Y}}

\def\bfmhhaY{\bfmath{\hhaY}} %\widehat{\widehat{Y}}}}
\def\bfmhhaY{\hbox to 0pt{$\widehat{\bfmY}$\hss}\widehat{\phantom{\raise 1.25pt\hbox{$\bfmY$}}}}

\def\haf{{\hat f}}
\def\hag{{\hat g}}

\def\haA{\widehat A}

\def\tiltheta{{\tilde \theta}}

\def\tilf{\tilde f}

\def\clE{{\cal E}}

\def\clW{{\cal W}}

 \def\head#1{\paragraph{\textit{#1}}}

\def\eqdef{\mathbin{:=}}

\def\Expect{{\sf E}}

\def\lgmath#1{{\mathchoice{\mbox{\large #1}}%
{\mbox{\large #1}}%
{\mbox{\tiny #1}}%
{\mbox{\tiny #1}}}}

\def\Zero{{\mathchoice{\lgmath{\sf 0}}%
{\mbox{\sf 0}}%
{\mbox{\tiny \sf 0}}%
{\mbox{\tiny \sf 0}}}}

\def\ind{\bbbone}
  
 \def\epsy{\varepsilon}

\def\varble{\,\cdot\,}

\def\formtmp#1#2{{\vskip12pt\noindent\fboxsep=0pt\colorbox{#1}{\vbox{\vskip3pt\hbox to \textwidth{\hskip3pt\vbox{\raggedright\noindent\textbf{#2\vphantom{Qy}}}\hfill}\vspace*{3pt}}}\par\vskip2pt%
\noindent\kern0pt}}

\titleformat\subparagraph[runin]
                        {\normalfont\normalsize\bfseries}
                        {mypar}
                        {0pt}
                        {}{}
\titlespacing\subparagraph{0pt}%%             left margin spacing
                       {.1ex minus 0.2ex}%% spacing above the \subparagraph
                       {.75em}%%             horizontal offset from title (horizontal because the `runin` shape was used.)

\newenvironment{programcode}[1]{\ignorespaces\def\stmtopen##1{##1}%
\pagebreak[3]%
\formtmp{programcode}{#1}%
\endlinechar=-1\relax%
\nopagebreak[4]}{%
\noindent\textcolor{programcode}{\rule{\columnwidth}{1pt}}\vskip1pt\par\addvspace{\baselineskip}%
\endlinechar=13}

\newtheoremstyle{thm}{12pt}{15pt}%
     {\itshape}%         Body font
     {}%         Indent amount (empty = no indent, \parindent = para indent)
     {\bfseries}% Thm head font
     {}%        Punctuation after thm head
     {0pt}%{8pt}%     Space after thm head (\newline = linebreak)
     {\thmname{#1}\thmnumber{ #2.}\thmnote{ \textbf{(#3)}}\quad}% Thm head spec

\def\bara{{\overline {a}}}

\def\barf{{\widebar{f}}}
\def\barg{{\widebar{g}}}

\def\barv{{\overline {v}}}

\def\barB{{\bar{B}}}

\def\barY{{\bar{Y}}}

{\end{list}}

\def\ass(#1:#2){(#1\ref{#1:#2})}

\def\ritem#1{
\item[{\sf \ass(\current_model:#1)}]
}

\newenvironment{recall-ass}[1]{% 
\begin{description}
\def\current_model{#1}}{
\end{description}
}

\def\sq{\hbox{\rlap{$\sqcap$}$\sqcup$}}
\def\qed{\ifmmode\sq\else{\unskip\nobreak\hfil
\penalty50\hskip1em\null\nobreak\hfil\sq
\parfillskip=0pt\finalhyphendemerits=0\endgraf}\fi}

\newcommand{\blot}{\vrule height 1.1ex width .9ex depth -.1ex }
\def\qedb{\ifmmode\blot\else{\vspace{-.2cm}\unskip\nobreak\hfil
\penalty50\hskip1em\null\nobreak\hfil\blot
\parfillskip=0pt\finalhyphendemerits=0\endgraf}\fi}

\newtheoremstyle{example}{15pt}{20pt}%
     {}%         Body font
     {}%         Indent amount (empty = no indent, \parindent = para indent)
     {\bfseries}% Thm head font
     {}%        Punctuation after thm head
     {1pt}%{8pt}%     Space after thm head (\newline = linebreak)
     {\thmname{#1}\thmnumber{ #2.}~\thmnote{\textit{\textbf{#3}}}%
     \\[.15cm]\unskip\nobreak}% Thm head spec

\theoremstyle{example}
\newtheorem{exmp}{Example}[section]

\theoremstyle{example}
\newcounter{rmnum}
\newenvironment{romannum}{\begin{list}{{\upshape (\roman{rmnum})}}{\usecounter{rmnum}
\setlength{\leftmargin}{18pt}
\setlength{\rightmargin}{8pt}
\setlength{\itemindent}{2pt}
}}{\end{list}}

\newcounter{anum}

\newcommand{\field}[1]{\mathbb{#1}}

\def\Re{\field{R}}

\def\nat{\field{Z}_+}

\def\Co{\field{C}}

\def\Expect{{\sf E}}

\def\transpose{{\intercal}}

\def\argmin{\mathop{\rm arg\, min}}
\def\ind{\hbox{\large \bf 1}}

\def\trace{\hbox{\rm trace\,}}  
\def\epsy{\varepsilon}
\def\varble{\,\cdot\,}

\def\haY{\widehat{Y}}

\def\hhaY{\hbox to 0pt{$\haY$\hss}\widehat{\phantom{\raise 1.25pt\hbox{Y}}}}

\def\haz{\widehat z}

\def\haA{\widehat A}

\def\haG{{\widehat G}}

\def\haY{\widehat Y}

\def\bfPhi{\bfmath{\Phi}}

\def\bfalpha{\bm{\alpha}}
\def\bfbeta{\bm{\beta}}

\newlength{\dhatheight}

\def\dlstep{r}

\def\scerror{Z}
\def\bfscerror{\bm{Z}}

\def\qsaDyn{\text{H}}

\renewcommand{\eqdef}{\ensuremath{\stackrel{\hbox{\sf\tiny def}}{=}}}

\def\Err{o}

\def\tilUpupsilon{\widetilde{\Upupsilon}}

\def\tilXi{\widetilde{\Xi}}

\def\prstate{\Upomega}

\def\XiI{\Xi^{\mathrm{\tiny I}}}

\renewcommand{\qsaprobe}{{\mathchoice{\mathlarger{\upxi}}%
		{\upxi}%
		{\upxi}%
		{\upxi}}}
\def\Ebox#1#2{%
\smallbreak \centerline{
\includegraphics[width= #1\hsize]{#2}  
}
}
 
 \theoremstyle{thm}
\newtheorem{theorem}{Theorem}[section]
\newtheorem{lemma}[theorem]{Lemma}
\newtheorem{proposition}[theorem]{Proposition}
\newtheorem{corollary}[theorem]{Corollary}

\usepackage[usenames,dvipsnames]{color}
\usepackage{color}
 \definecolor{programcode}{gray}{0.9}

\definecolor{MyDarkBlue}{cmyk}{0.5,0.1,0,0.9}

\usepackage[colorlinks,%
linkcolor=BrickRed,% 
filecolor =BrickRed,%
citecolor=RoyalPurple,%
]{hyperref}

\newlength{\noteWidth}
\long\def\notes#1{\ifinner
{\footnotesize #1}
\else 
\marginpar{\parbox[t]{\noteWidth}{\raggedright\footnotesize#1}}
\fi\typeout{#1}}

 \def\notes#1{\typeout{check notes!!!}}   %  For final version

\def\rd#1{{\color{red}#1}}

\def\sfb#1{}

\usepackage{cleveref}

\Crefname{corollary}{Corollary}{Corollaries}
\Crefname{eqnarray}{eq.}{eqs.}
\Crefname{equation}{eq.}{eqs.}

\Crefname{figure}{Fig.}{Figs.}
\Crefname{tabular}{Tab.}{Tabs.}
\Crefname{table}{Tab.}{Tabs.}
\Crefname{proposition}{Prop.}{Propositions}
\Crefname{theorem}{Thm.}{Thms.}
\Crefname{definition}{Def.}{Defs.} 
\Crefname{section}{Section}{Sections}
\Crefname{lemma}{Lemma}{Lemmas}
\Crefname{assumption}{Assumption}{Assumptions}

\title{\LARGE   Accelerating  Optimization and Reinforcement Learning
\\
 with  
Quasi-Stochastic Approximation  
}
\author{Shuhang Chen 
	\and Adithya Devraj 
	\and Andrey Bernstein 
	\and Sean Meyn% <-this % stops a space 
	\thanks{SC and SPM are with the University of Florida, Gainesville, FL 32611  (mathematics and ECE).  
		AD is with the EE department at Stanford University, Stanford, CA 94305.  AB is with the National Renewable Energy Laboratory, Golden, CO 80401.
		Financial support from ARO award W911NF1810334
		and National Science Foundation award  EPCN 1935389
		is gratefully acknowledged.}%
}
\date{}

\begin{document}
 
\maketitle

\begin{abstract} 
 
The ODE (ordinary differential equation) method has been a workhorse for algorithm design and analysis since the introduction of the stochastic approximation technique of Robbins and Monro in the early 1950s.    It is now understood that convergence theory amounts to establishing robustness of Euler approximations for ODEs,  while theory of rates of convergence requires finer probabilistic analysis. This paper sets out to extend this theory to quasi-stochastic approximation (QSA), based on algorithms in which the ``noise'' or ``exploration'' is based on deterministic signals, much like quasi-Monte Carlo.   The main results are obtained under minimal assumptions:  the usual Lipschitz conditions for ODE vector fields,   and for rate results it is assumed that there is a well defined linearization near the optimal parameter $\theta^\ocp$, with Hurwitz linearization matrix $A^\ocp$.    Algorithm design is performed in continuous time,  in anticipation of discrete-time implementation based on Euler approximations, or high-fidelity alternatives.

The main contributions are summarized as follows:
 \begin{romannum}
 
\item  If the algorithm gain is chosen as $a_t = g/(1+t)^\rho$ with $g>0$ and $\rho\in(0,1)$,  then the rate of convergence of the algorithm is $1/t^\rho$.   There is also a well defined ``finite-$t$'' approximation:
\[
a_t^{-1} \{ \ODEstate_t - \theta^\ocp \}   =   \barY + \XiI_t + o(1) 
\]
where $\barY\in\Re^d$ is a vector identified in the paper, and $\{ \XiI_t   \}$ is bounded with zero temporal mean.

\item With gain   $a_t = g/(1+t)$ the results are not as sharp:   the rate of convergence  $1/t$  holds only if $I + g A^\ocp$ is Hurwitz.   Hence we obtain the optimal rate of convergence only if $g>0$ is chosen sufficiently large.

\item  Based on the Ruppert-Polyak averaging technique of stochastic approximation, one would expect that  a convergence rate of $1/t$ can be obtained by averaging: 
\[ 
\ODEstate^{\text{RP}}_T 
	 =
  \frac{1}{T}  \int_{0}^T \ODEstate_t \, dt  
\]
where the estimates $\{  \ODEstate_t  \}$ are obtained using the gain in (i).   The preceding sharp bounds imply that averaging results in $1/t$ convergence rate if and only if $\barY = \Zero$.   This condition holds if the noise is  additive,  but appears to fail in general.

\item   The theory is illustrated with applications to gradient-free optimization,  and   policy gradient algorithms for reinforcement learning.

\end{romannum}

\medskip
 \noindent
\textbf{Note:}
This pre-print is written in a tutorial style so it is accessible to new-comers.  It will be a part of a handout for upcoming short courses on RL.   A more compact version suitable for journal submission is in preparation.

\smallskip

%49L20  	Dynamic programming in optimal control and differential games
%%93E20  	Optimal stochastic control
%68T05  	Learning and adaptive systems in artificial intelligence
%93E35  	Stochastic learning and adaptive control
%%%% 93B47  	Iterative learning control

\noindent
%AMS-MSC: 68T05 (Primary), 93E35, 49L20 (Secondary)
AMS-MSC: 68T05, 93E35, 49L20

\end{abstract}

\clearpage

\thispagestyle{empty} 
%\setcounter{page}{0}
%

 %%%%%%%%%%%%%%%%%%%%%%%%%%%%%%%%%%%%%%%%%%%%%%%%%%%%%%%%%%%%%%%%%%%%%%%%%%%%%%%%

\clearpage
\section{Introduction} 
\label{s:intro}

%\section{Quasi-Stochastic Approximation}
%\label{s:QSA}

  \sfb{
This section is an early introduction to \Cref{s:SA}, which concerns ODE approximations and algorithm design in a stochastic setting, based on the theory of \textit{stochastic approximation}.  
  \mindex{Quasi-Stochastic Approximation}}

The \textit{ODE method} was coined by Ljung in his 1977 survey of stochastic approximation (SA) techniques for analysis of recursive algorithms \cite{lju77}.    The theory of SA was born more than 25 years earlier with the publication of the work of Robbins and Monro~\cite{robmon51a}, and research in this area has barely slowed in the 70 years since its publication.  
The goal of  SA  is a simple root finding problem:   compute or approximate the vector $\theta^\ocp\in\Re^d$ solving  $\barf(\theta^\ocp) = \Zero$,   in which $\barf\colon\Re^d\to\Re^d$ is defined by an expectation:   
\begin{equation}
\barf(\theta) \eqdef \Expect[f(\theta,\Phi)]\,,\qquad \theta\in\Re^d\,,  
\label{e:barf}
\end{equation} 
where $f\colon\Re^d \times \prstate \to \Re^d$, and $\Phi$ is  a random vector  taking values in a set $\prstate$ (assumed in this paper to be a subset of Euclidean space).  

Given the evolution of this theory, it is useful to reconsider the meaning of the method.   Rather than merely a method for analyzing an algorithm, the ODE method is an approach to algorithm design, broadly described in two steps:
\begin{romannum}
\item[]\textbf{Step 1:}     Design $f$ (and hence $\barf$) so that the following ODE is globally asymptotically stable: 
\begin{equation}
\ddt \odestate_t  = \barf  ( \odestate_t  ).
\label{e:ODE_SA}
\end{equation}

\item[]\textbf{Step 2:}   Obtain a discrete-time algorithm via a ``noisy'' Euler approximation:
\begin{equation} 
\theta_{n+1} = \theta_n +\alpha_{n+1} [ \barf(\theta_n) + \tilXi_n ]  \, , \qquad n\ge 0,
\label{eq:SA0}
\end{equation}
in which $\{\alpha_n\}$ is the non-negative ``gain'' or ``step-size'' sequence, and the sequence $\{ \tilXi_n \}$ has mean that vanishes as $n\to\infty$.
\end{romannum}
Each step may require significant ingenuity.   Step~1 may be regarded as a control problem: design dynamics to reach a desirable equilibrium from each initial condition.   
There are algorithm design questions in Step~2, but in the original formulation of SA and nearly all algorithms that follow,   recursion \eqref{eq:SA0} takes the form  
\begin{equation} 
\theta_{n+1} = \theta_n +\alpha_{n+1} f(\theta_n,\Phi_{n+1})\, , \qquad n\ge 0,
\label{eq:SA}
\end{equation} 
in which  $ \Phi_{n+1}$ has the same distribution as $\Phi$, or the  distribution of $\Phi_{n+1}$ converges to that of $\Phi$ as $n\to\infty$.   A useful approximation   requires assumptions on $f$, the ``noise'' $\Phi_{n+1}$,   and the step-size sequence $\bfma$.   The required assumptions, and the mode of analysis, are  not very different than what is required to successfully apply a deterministic Euler approximation \cite{bor20a}. 

The motivation for the abstract formulation of Step~2 is to reinforce the idea that we are not bound to the traditional recursion.  
 For example,  control variate techniques offer alternatives:  an example is the introduction of the ``advantage function'' in policy gradient techniques for reinforcement learning (RL)
\cite{sutetall00,sutbar18}.  \notes{Probably need to reference others:  \cite{baxbar01}?    Maybe even \cite{jaajorsin94a}}

\sfb{It will be seen in  \Cref{s:SA} that this recursion does approximate the associated ODE:   }

  \sfb{
\textit{For those of you with a background in probability:}
...
variation of the Euler scheme \eqref{e:ODE_Euler},
}

\sfb{ \eqref{e:ODE_Euler}.
}

\sfb{ a deterministic state space model always satisfies the \textit{Markov property} used in Part~2 of this book.    For example,   for given $\omega>0$, the sequence $\Phi_{n+1}= [ \cos(\omega n), \sin(\omega n) ]$  is a Markov chain on the unit circle in $\Re^2$.  
\\
 The upshot of stochastic approximation is that it can be implemented  without knowledge of the function $f$ or of the distribution of $\Phi$; rather, it can rely on observations of the sequence $\{ f(\theta_n,\Phi_{n+1}) \}$. This is one reason why these algorithms are valuable in the context of reinforcement learning.}
\sfb{:  a topic considered in depth in \Cref{s:MC}. }
 
 \notes{
  (RL)~\cite{bor20a,bertsi96a,huachemehmeysur11,devmey17a,devmey17b}.
In such cases, the driving noise is typically modeled as a Markov chain.
}

\notes{boring, and not 'key':
A key observation in this paper is that Markov chains need not be stochastic:    For example,   for given $\omega>0$, the sequence 
$\Phi_{n+1}= [ \cos(\omega n), \sin(\omega n) ]$  is a Markov chain on the unit circle in $\Re^2$.    
}

In much of the SA literature, and especially in the applications considered in this paper,  it is assumed that $\bfPhi$ is independent and identically distributed (i.i.d.), or more generally, it is a \textit{Markov chain}.  
Just as quasi Monte-Carlo algorithms are motivated by fast approximation of integrals,  
the \emph{quasi-stochastic approximation} (QSA) algorithms considered in this paper are designed to speed convergence for root finding problems.   It is useful to pose the algorithm in continuous time:
\begin{equation}
\ddt\ODEstate_t = a_t f(\ODEstate_t,\qsaprobe_t) \,.
\label{e:QSAgen}
\end{equation}
where in the main results we restrict to gains of the form
\begin{equation}
a_t = g/(1+t)^\rho\,,  \qquad \text{ with $0 < \rho\le 1$  and $g>0$}
\label{e:gainQSA}
\end{equation}
The \textit{probing signal}   $\bfqsaprobe$ is generated from a \emph{deterministic} (possibly oscillatory) signal rather than a stochastic process.   
\mindex{Probing signal} 
 Two canonical choices   are the $m$-dimensional mixtures of periodic functions:   
\begin{subequations}
\begin{align}
\qsaprobe _t  &= \sum_{i=1}^K  v^i   [   \phi_i + \omega_i t      ]_{   \text{(mod $1$)}}
\label{e:ProbeSawQSA}
\\
\qsaprobe _t  &= \sum_{i=1}^K  v^i   \sin (2\pi [ \phi_i +  \omega_i t ] )
\label{e:ProbeSinuQSA}
\end{align}%
\label{e:SawSinusoidalQSA}%
\end{subequations}%
for fixed vectors $\{v^i\} \subset\Re^m$, phases $\{\phi_i\}$,   and  frequencies   $\{\omega_i\}$. 
Under mild conditions on $f$ we can be assured of the existence of this limit defining the mean vector field:
\begin{equation}
\barf(\theta) = \lim_{T\rightarrow\infty}\frac{1}{T}\int_0^T f(\theta,\qsaprobe_t)\, dt,\ \ \textrm{for all }\theta\in\Re^d.
\label{e:ergodicA1}
\end{equation}
Our aim is to obtain tight bounds between solutions of \eqref{e:QSAgen}  and 
 \begin{equation}
\ddt \barODEstate_t = a_t \barf(\barODEstate_t)  \,,\qquad t\ge t_0  \,,  \    \barODEstate_{t_0} =   \ODEstate_{t_0}
\label{e:ODE_haSA}
\end{equation}
where the choice of $t_0$ depends on the stability properties of the associated ODE \eqref{e:ODE_SA}  with constant gain.

\paragraph{Contributions} 

It is assumed throughout the paper that \eqref{e:ODE_SA} is globally asymptotically stable, with unique equilibrium denoted  $\theta^\ocp \in\Re^d$.   Convergence of \eqref{e:QSAgen} is also assumed: sufficient conditions based on the existence of a Lyapunov function can be found in  \cite{berchecoldalmehmey19a,berchecoldalmehmey19b}.  This prior work is reviewed in the Appendix, along with a new sufficient condition extending the main result of  \cite{bormey00a}. 

We say that the rate of convergence of  \eqref{e:QSAgen} is $1/t^{\varrho_0}$ if
\begin{equation}
\limsup_{t\to\infty} t^\varrho \| \tilODEstate_t  \|  
	=  \begin{cases}
	\infty  &   \varrho>\varrho_0
	\\
	0    & \varrho< \varrho_0
	  \end{cases}
\label{e:QSArate}
\end{equation}
where  $\tilODEstate_t \eqdef \ODEstate_t - \theta^\ocp $ is the estimation error.      By careful design we can achieve $\varrho_0 =1$, which is optimal in most cases (such as for Monte-Carlo -- see \Cref{s:QMC}).    

Solutions to \eqref{e:ODE_haSA} and \eqref{e:ODE_SA} are related by a time-transformation (see \Cref{t:barODEstate}---a common transformation in the SA literature), so that $\barODEstate_t \to\theta^\ocp$.    Conditions are imposed so the rate of convergence is faster than $1/t$, which justifies the following scaling:    
\begin{equation}
\scerror_t = \frac{1}{a_t}  \bigl(  \ODEstate_t -  \barODEstate_t \bigr) \, ,   \qquad t\ge t_0
\label{e:scaled_error}
\end{equation}
We obtain general conditions under which $\bfscerror$ is bounded and non-vanishing using the gain
\eqref{e:gainQSA}, which then implies the rate of convergence of \eqref{e:QSAgen} is $1/t^\rho$.
A key assumption is that the ODE is smooth near the equilibrium $\theta^\ocp$,   so that there is a well defined \textit{linearization matrix}
 $A^\ocp =\partial\barf\, (\theta^\ocp)$.   
 
Given this background, we are ready to summarize the main results:

\noindent
\textbf{1.} \  
The following dichotomy is established:
 \begin{romannum}
\item  With $\rho<1$ we obtain $1/t^\rho$ rate of convergence, provided  $A^\ocp  $ is Hurwitz, regardless of the value of $g>0$ appearing in \eqref{e:gainQSA}.

\item  With $\rho=1$ we obtain the optimal rate of convergence $1/t$  provided  $I+g A^\ocp  $ is Hurwitz. 
\end{romannum}
In either case, an exact approximation of the scaled error is obtained:
\begin{equation}
  \scerror_t    = \barY  +  \XiI_t   + o(1)  
\label{e:scaledError1}
\end{equation} 
where $\barY\in\Re^d$ is a vector identified in \eqref{e:barY}, and 
\[
 \XiI_t  = \int_0^t   f(\theta^\ocp,\qsaprobe_r) \, dr 
\]
This is assumed bounded in $t$ (justified for natural classes of probing signals in  \Cref{s:quasiMarkov}).

\noindent
\textbf{2.} \   For those well-versed in stochastic approximation theory,  the preceding conclusions would motivate the use of averaging:
 \[
\ODEstate^{\text{RP}}_T  =   \frac{1}{T} \int_0^T     \ODEstate_t\,  dt\,,  \qquad T > 0
\]
where $\{  \ODEstate_t \}$ are obtained using the gain \eqref{e:gainQSA} with $\rho<1$.  The superscript refers to the averaging technique for stochastic approximation introduced independently by Ruppert and Polyak  \cite{rup88,poljud92}.

  In general, this approach \textit{fails}:  the rate of convergence of $\{\ODEstate^{\text{RP}}_T \}$ to $\theta^\ocp$ is $1/T$ if and only if $\barY =\Zero$. A sufficient condition for this is additive noise, for which \eqref{e:QSAgen} becomes
\begin{equation}
\ddt\ODEstate_t = a_t \{ \barf(\ODEstate_t)   +   D(\qsaprobe_t) \}
\label{e:QSAadditiveNoise}
\end{equation}
In this case    
\[
\XiI_t  = \int_0^t   D(\qsaprobe_r) \,, dr 
\]

 \noindent
\textbf{3.} \ 
These theoretical results motivate application to gradient-free optimization,  and policy-gradient techniques for RL.     
Theory and examples are presented in \Cref{s:extremeQSA,s:ActorCriticQSA}.

\paragraph{Literature review}

This paper spans three areas of active research:    %  SA, gradient free optimization, and policy gradient techniques for RL.  

\noindent\textbf{1.  Stochastic approximation}
  \Cref{s:ex} and some of the convergence theory in the appendix  is adapted from \cite{berchecoldalmehmey19a,berchecoldalmehmey19b},  which was inspired by the prior results in  \cite{mehmey09a,shimey11};
\cite{cheberdevmey20} contains applications to gradient-free optimization with constraints.     The first appearance of QSA methods appears to have originated in the domain of quasi-Monte Carlo methods applied to finance; see~\cite{lappagsab90,larpag12}.  

The contributions here are a significant extension of   \cite{berchecoldalmehmey19a,berchecoldalmehmey19b},  which considered exclusively 
  the special case in which the function $f$ is linear, with  $f(\theta,\qsaprobe) = A\theta + B \qsaprobe$ for matrices $A,B$.   The noise is thus additive:  \eqref{e:QSAadditiveNoise} holds with $ D(\qsaprobe_t)  =   B \qsaprobe_t$.   Using the gain \eqref{e:gainQSA} with $\rho=1$, the optimal rate of convergence was obtained under the assumption that  $I+g A  $ is Hurwitz.    
The assumptions on $g$ using $\rho=1$ are stronger than what is imposed in stochastic approximation, which requires that $\half I+g A  $ is Hurwitz.  
On the other hand, the conclusions for stochastic approximation algorithms are weaker:   from \eqref{e:scaledError1}
we obtain
\[
  t^2 \|  \ODEstate_t  - \theta^\ocp \|^2    =   \|  \barY  +  \XiI_t   + o(1)   \|^2   
\]
In the theory of stochastic approximation we must introduce an expectation,  and settle for a much slower rate: 
\begin{equation}
\lim_{t\to\infty} \, t \Expect[\|\ODEstate_t  - \theta^\ocp \|^2 ]   =  \trace( \Sigma_\theta )  
\label{e:avar}
\end{equation}
That is, the rate is $1/\sqrt{t}$ rather than $1/t$:  see \cite{kalmounautadwai20,chedevbusmey20} for refinements of this result and history.

\sfb{
Ruppert-Polyak averaging was introduced independently in \cite{rup88,poljud92}.  However, this work has nothing to do with QSA, but  concerns optimizing the covariance $\Sigma_\theta$ appearing in 
 \eqref{e:avar} for stochastic approximation.}

  The ``ODE@$\infty$'' \eqref{e:barfinfty_QSA}  was introduced in  \cite{bormey00a} for stability verification in stochastic approximation.    \Cref{t:convergenceBM_QSA} is an extension of the Borkar-Meyn Theorem     \cite{bormey00a,bor20a}, which
 has been refined considerably in recent years    \cite{rambha17,rambha18}.

\noindent\textbf{2. Gradient free optimization}  The goal is to minimize a loss function $\Obj(\theta)$ over $\theta\in\Re^d$.   It is possible to observe the loss function at any desired value, but no gradient information is available.  

Gradient-free optimization has been studied in two, seemingly disconnected lines of work:    techniques intended to directly approximate gradient descent through perturbation techniques, known as Simultaneous Perturbations Stochastic Approximation (SPSA),   and    Extremum-Seeking Control (ESC)  which is formulated in a purely deterministic setting.
The contributions of the present paper are motivated by both points of view, but leaning more on the former.    See \cite{EShistory2010,liukrs12,arikrs03} for the nearly century-old history of ESC.
\notes{AD: Latter?}

% \notes{kiewol56 is a queueing paper!}

\notes{AD: Spall coined the term SPSA..}
Theory for SPSA began with the algorithm of  Keifer-Wolfowitz \cite{kiewol52},  which requires at each iteration access to two perturbations per dimension to obtain a stochastic gradient estimate.   
This computational barrier was addressed in the work of Spall \cite{spa87,spa92,spa97}.   
Most valuable for applications in RL is the one-measurement form of SPSA introduced in \cite{spa97}:
this can be expressed in the form \eqref{eq:SA}, in which
\begin{equation}
f(\theta_n,\Phi_{n+1})   =    \Obj(\theta_n + \epsy \Phi_{n+1})\Phi_{n+1}
\label{e:one-meas-SPSA}
\end{equation} 
where $\bfPhi$ is a zero-mean and i.i.d.\ vector-valued sequence.    The
qSGD (quasi-Stochastic Gradient Descent) algorithm  \eqref{e:ES} is a continuous time analog of this approach.

The introduction of \cite{nesspo17} suggests that there is an older history of improvements to SPSA in the Russian literature:   see eqn.~(2) of that paper and surrounding discussion.
Beyond history, the contributions of  \cite{nesspo17} include rates of convergence results for standard and new SPSA algorithms. Information theoretic lower bounds for optimization methods that have access to noisy observations of the true function was derived in \cite{jamnowrec12}. This class of algorithms also has some bandits history \cite{agarwal2010optimal,bubeck2012regret}.
%\notes{AD: OK?}
%\bl{ Not sure where to put this:
%... In optimization literature, \cite{nesspo17} consider deterministic convex optimization problems (with both smooth and non-smooth objectives), and prove convergence rates of algorithms that use randomized perturbations.}

In all of the SPSA literature surveyed above, a gradient approximation is obtained through the introduction of an i.i.d.\ probing signal.  For this reason, the best possible rate is of order $1/\sqrt{n}$.   
The algorithms introduced in the present work are designed to achieve $O(1/n)$ convergence rate for optimization and root-finding problems.

%\notes{The below paper is interested in stochastic optimization.}
More closely related to the present work is \cite{bhafumarwan03} which treats SPSA using a specially designed class of deterministic probing sequences.  
There are no comparable contributions,  but this previous work motivates further research on rates of convergence for the algorithms proposed.

% they consider a deterministic version of SPSA, wherein the perturbations are deterministic sequences, as in this paper. But they consider very specific sequences: lexicographic and normalized Hadamard matrix based deterministic perturbations. There is also some two-time-scale theory in there, because they need to freeze the parameter updates before they collect all the required measurements for the noisy gradient estimate (maybe important because they call their algorithm two-time-scale-SPSA). 

% \notes{How possible for a vector? "... but once again replacing $\qsaprobe_t$ with $\qsaprobe_t^{-1}$"
% \\
% I guess we don't need this:
%\rd{\cite{ger99} is probably most relevant. Based on the abstract, they derive convergence rates of Spall's SPSA algorithms for stochastic optimization. They claim to show convergence rates of $O(n^{-\beta/2})$, with best $\beta$ being $2/3$. However, for Newton-Raphson based SPSA algorithm, the claim is that $\beta$ can be made arbitrarily close to $1$. The rates derived are CLT kind of rates.}
%}

%\rd{What about 
%\cite{bhabor03} ?}
%\notes{AD: They consider stochastic pertubations SPSA algorithm for a specific problem. Don't think it is relevant, given the number of papers we have cited.}

% \notes{AD: Sounds like you are saying all policy gradient techniques can be regarded as gradient-free methods?}
\noindent\textbf{3. Policy gradient techniques}
These may be regarded as a special case of gradient-free optimization, in which the goal is to minimize average cost for an MDP based solely on input-output measurements. 
The standard dynamic programming formulation for optimal control is replaced with the following architecture:
given a parameterized family of (possibly randomized) state-feedback polices $\{ \fee^\theta  :  \theta\in\Re^d \}$,   the goal is to minimize the associated average cost $\Obj(\theta) = \Expect_\theta[c(X_k,U_k)]$, where the expectation is in steady-state, subject to $U_k = \fee^\theta(X_k)$ for all $k$  (the state description must be extended if the policy is randomized).  
Williams' REINFORCE \cite{wil92} is an early example,   while the most popular algorithms today are of the ``actor-critic'' category in which the approximation of the gradient of $\Obj$ is based on an estimate of a value function.   In the widely cited recent paper \cite{mania2018simple} it is shown that SPSA algorithms such as a modified version of \eqref{e:one-meas-SPSA} can sometimes outperform actor-critic methods.   
%\notes{AD: Again, just to clarify, SPSA is Spall's algorithm. Ben Recht doesn't cite any of Spall's papers. He just cites Nesterov.. But I think the statement is correct.}

  \smallskip
  
\noindent\textbf{Organization}  
\Cref{s:ex} contains some simple examples, introduced mainly to clarify notation and motivation,  much of which is adapted from  \cite{berchecoldalmehmey19a,berchecoldalmehmey19b}.  
The main results are summarized in  \Cref{s:GainQSA}, along with sketches of the proofs.   Applications to 
gradient-free optimization are summarized in \Cref{s:extremeQSA},  and to policy gradient RL in
\Cref{s:ActorCriticQSA}.
\Cref{s:conc} contains conclusions and directions for future research.
The appendix contains three sections:  
\Cref{s:QSAconvergence} concerns convergence theory of  QSA based on   \cite{berchecoldalmehmey19a,berchecoldalmehmey19b}.  The main challenge is to establish boundedness
of trajectories of the algorithm.  A new sufficient condition is established, based on a generalization of the Borkar-Meyn Theorem.   
Ergodic theory for dynamical systems is  required in the main assumptions:  justification for a broad class of ``probing signals'' is contained in \Cref{s:quasiMarkov}.  
\Cref{s:tech}  contains the details of the proofs of the main results.

\section{Simple Examples}
\label{s:ex}

 The following subsections contain examples to illustrate theory of QSA, and also a glimpse at applications.   
\Cref{s:QMC,s:QSARL} are adapted from   \cite{berchecoldalmehmey19a,berchecoldalmehmey19b}.
%\notes{Section copied/pasted from Chapter 4}

\subsection{Quasi Monte-Carlo}
\label{s:QMC}

Consider the problem of obtaining the integral over the interval $[0,1]$ of a function $y\colon\Re\to\Re$.  
In a standard Monte-Carlo approach we would draw independent random variables $\{\Phi_{n+1}\}$,  with distribution uniform on the interval $[0,1]$, and then average:
\begin{equation}
\theta_n = \frac{1}{n} \sum_{k=0}^{n-1} y(\Phi_k )
\label{e:MC_standard}
\end{equation}
A QSA analog is described as follows:   the probing signal is the one-dimensional \textit{sawtooth function},  $\qsaprobe_t \eqdef t $ (modulo 1)  and consider the analogous average 
\begin{equation}
\ODEstate_t = \frac{1}{t} \int_0^t  y(\qsaprobe_r )\,  dr
\label{e:MCQSA}
\end{equation}

Alternatively, we can adapt the QSA model \eqref{e:QSAgen} to this example, with
\begin{equation}
f(\theta, \qsaprobe) \eqdef y(\qsaprobe) - \theta.
\end{equation}
The averaged function is then given by   
\begin{align*}
\barf (\theta) =
 \lim_{T\rightarrow\infty}\frac{1}{T}\int_0^T f(\theta,\qsaprobe_t )\, dt 
 = \int_0^1 y(\qsaprobe_t)\, dt  - \theta  
\end{align*}
so that  $\theta^\ocp = \int_0^1 y(\qsaprobe_t) \, dt $ is the unique root of $\barf$.  Algorithm \eqref{e:QSAgen} is given by:
\begin{equation}
 \ddt \ODEstate_t = a_t [y(\qsaprobe_t) - \ODEstate_t ].
 \label{e:QMC}
\end{equation}

\begin{figure}[ht] 
\Ebox{.5}{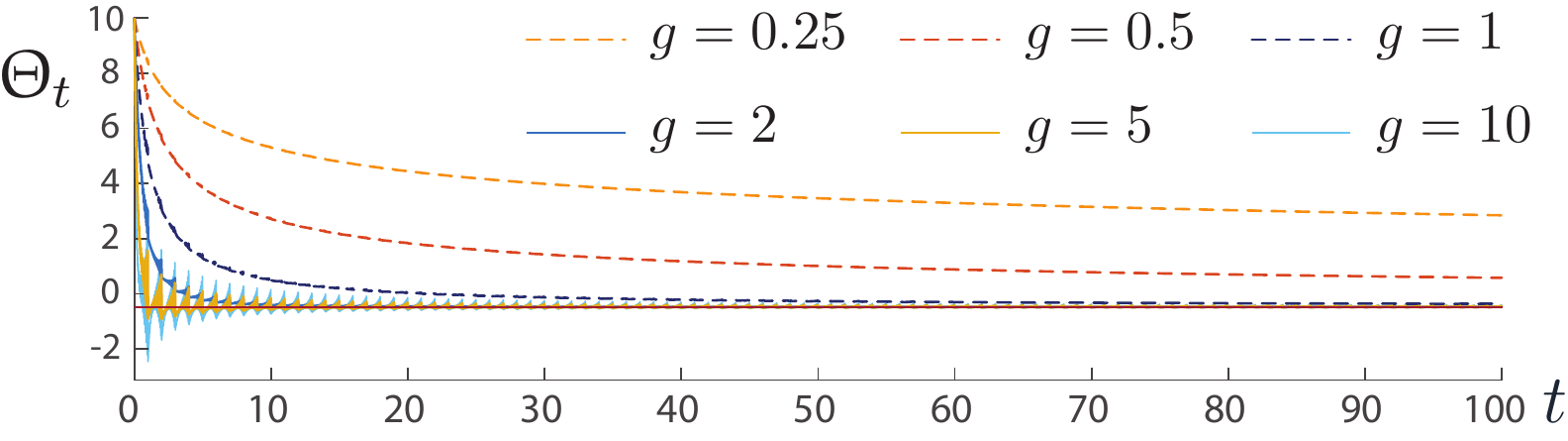} 
\caption{Sample paths of Quasi Monte-Carlo estimates.} 
\vspace{-.5em}
\label{f:plotsQMC}
\end{figure}

This Monte-Carlo approach \eqref{e:MCQSA} can  be transformed into something resembling  \eqref{e:QMC}.   
Taking derivatives of each side of \eqref{e:MCQSA}, we obtain using the product rule of differentiation, and the fundamental theorem of calculus,
\[
 \ddt \ODEstate_t    =    - \frac{1}{t^2}\int_0^t  y(\qsaprobe_r )\,  dr  + \frac{1}{t}  y(\qsaprobe_t )    =  \frac{1}{t} [y(\qsaprobe_t) - \ODEstate_t ]
\]
This is precisely  \eqref{e:QMC} with $a_t = 1/t$   (not a great choice for an ODE design, since it is not bounded as $t\downarrow 0$).

The numerical results that follow are based on   $ y(\theta) =  e^{4t}\sin(100 \theta) $, whose mean is   $\theta^\ocp \approx   -0.5$. The differential equation  \eqref{e:QMC}  was approximated using a standard Euler scheme with sampling interval $10^{-3}$.      Several variations were simulated, differentiated by the gain $a_t=g/(1+t)$.  \Cref{f:plotsQMC}  shows typical sample paths of the resulting estimates for a range of gains,  and common initialization $\ODEstate_0=10$.
In each case, the estimates converge to the true mean  $\theta^\ocp \approx   -0.5$, but convergence is very slow for  $g>0$ significantly less than one.
Recall that   the case $g=1$ is very similar to what was obtained from the
Monte-Carlo approach \eqref{e:MCQSA}.

\notes{This exotic function was among many tested -- it is used here only because the conclusions are particularly striking.  
}

Independent trials were conducted to obtain variance estimates.    In each of $10^4$ independent runs, the common initial condition was drawn from $N(0,10)$,  and the estimate was collected at time $T=100$. 
\Cref{f:hists}  shows three histograms of estimates for standard Monte-Carlo \eqref{e:MC_standard}, and QSA using gains $g=1$ and $2$.   
An alert reader must wonder:   \textit{why is the variance reduced by 4 orders of magnitude when the gain is increased from $1$ to $2$?}  The relative success of the high-gain algorithm is explained in  \Cref{t:Couple_main}.

\begin{figure}[ht]
\Ebox{.85}{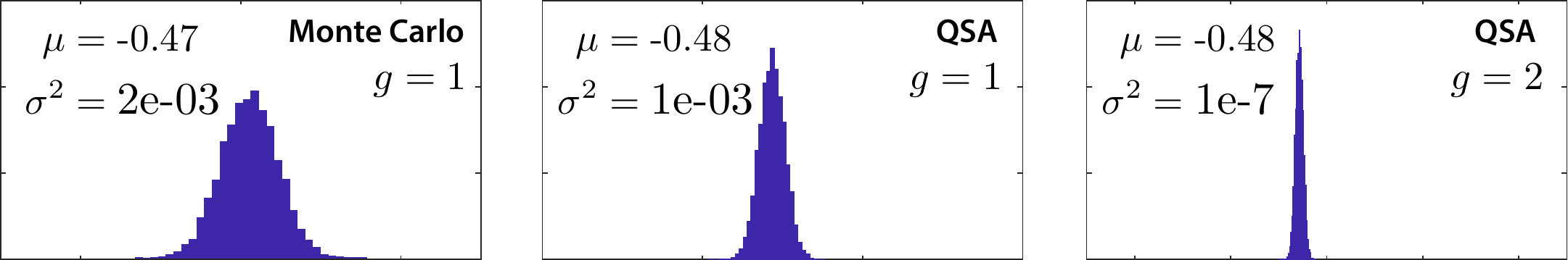} 
\caption{Histograms of Monte-Carlo and Quasi Monte-Carlo estimates after $10^4$ independent runs.   The optimal parameter is $\theta^\ocp \approx   -0.4841$. }  
\label{f:hists}
\end{figure}

\notes{see commented text.   Surely the BM theorem can be used}

%Yue, see figures attached — this may be your very first experiment!  I wish you had made it a bit trickier (a true local min at -1 rather than a saddle point), but this will do.  
%
%Last night I thought I would do a full “analysis” for the case J(theta) = theta^2.   That is, write down all the equations and see if I could make it clear to an elementary school kid that this should work.     Just the opposite happened.  For both of our quasi-gradient descent algorithms, the RHS of the algorithm is quadratic  in the state variables, so I expect we might see real problems with Euler.   Did you ever see “explosion” for large initial conditions?
%
%I want to sort out the theory.   Did we prove ultimate boundedness if J is Lipschitz continuous?   I’m hoping the “Borkar Meyn Theorem” will come to the rescue, since our Lyapunov function approach is probably harder to verify.   

\notes{commented out "Buyer beware" since I believe we can prove the results in the discrete time case.  I have added some comments in the lit review}

\subsection{Constant gain algorithm}

In later sections we will consider the linear approximation:
\begin{equation}
f(\theta,\qsaprobe) = A(\theta - \theta^\ocp)  +  B \qsaprobe
\label{e:QSAlinear}
\end{equation}
This provides insight, and sometimes we can show strong coupling between the linear and nonlinear QSA ODEs.  
We briefly consider here this linear model   in which  $a_t =\alpha$ is constant. 
Then QSA is a time-invariant linear system:
\[
\ddt \ODEstate_t   = \alpha [ A\tilODEstate_t + B \qsaprobe_t  ] \,,\qquad \tilODEstate_0=\tiltheta_0 
\]
where  $\tilODEstate_t \eqdef \ODEstate_t - \theta^\ocp $ is the error at time $t$.
For this simple model we can  solve the ODE when the probing signal is the mixture of sinusoids \eqref{e:ProbeSinuQSA}.

A linear system satisfies the \textit{principle of super-position}.    To put this to work,   consider the probing signal  \eqref{e:ProbeSinuQSA},   and for each $i$,  consider the ODE  
\[
\ddt \tilODEstate_t^i   = \alpha \bigl[ A\tilODEstate_t +  B v^i   \sin (2\pi [ \phi_i +  \omega_i t  ] ) \bigr)  \,,\qquad \tilODEstate_0^i =0  
\]
The principle states that the solution to the ODE is the sum:
\begin{equation}
\tilODEstate_t   =  e^{\alpha A t}  \tiltheta_0   + B\sum_{i=1}^K \tilODEstate_t^i  
\label{e:ODEstateLinear_soln}
\end{equation}
We see that the response to the initial error $\tiltheta_0 = \theta_0 - \theta^\ocp$ decays to zero exponentially quickly. Moreover, to  understand the steady-state behavior of the algorithm it suffices to fix a single value of $i$.

\notes{SM and SC both don't understand SM's statement: 
For more complex probing signals we can again justify consideration of sinusoids, provided we can justify  a Fourier series approximation.
}

Let's keep things simple, and stick to sinusoids.  And it is much easier to work with complex exponentials:  
 \[
\ddt \tilODEstate_t  = \alpha [ A\tilODEstate_t +  B v  \exp(j \omega t)   ]   \,,\qquad \tilODEstate_0=0  
\]
with $\omega\in\Re$ and $v\in\Re^d$  (dropping the scaling $2\pi$ for simplicity, and the phase $\phi$ is easily returned by a time-shift).   We can express the solution as a convolution:
\[
\begin{aligned}
 \tilODEstate_t  &= \alpha   \int_0^t  \exp\bigl( \alpha A r  \bigr)   B v  \exp\bigl(  j \omega (t-r)  \bigr) \,  dr
   \\
            & =  \alpha  \Bigl(   \int_0^t  \exp\bigl( [\alpha A  -j \omega  I] r  \bigr) \,  dr  \Bigr)  B  v  \exp\bigl(  j \omega t \bigr) 
     \end{aligned} 
\]
Writing $D= [\alpha A  -j \omega  I] $, the integral of the matrix exponential is expressed,  
\[
  \int_0^t  e^{D r} \, dr = D^{-1}  \bigl[e^{D t} -  I \bigr] 
\]
Using linearity once more,  and the fact that the imaginary part of $e^{j\omega t}$ is $\sin (\omega t)$, we arrive at a complete representation for  \eqref{e:ODEstateLinear_soln}: 
\begin{proposition}
\label[proposition]{t:constantGainQSAsoln}
Consider the  linear model with $A$ Hurwitz, and probing signal  \eqref{e:ProbeSinuQSA},  for which the constant-gain QSA algorithm has the solution \eqref{e:ODEstateLinear_soln}.    
Then $\tilODEstate_t^i  = \alpha \Upgamma_t^i    v^i  $ for each $i$ and $t$, with
\begin{equation}
\Upgamma_t^i   =     \mathrm{Im} \Bigl(   [\alpha A  -j \omega  I] ^{-1}   \bigl[  \exp\bigl( \alpha At  \bigr)  -  \exp\bigl(  2\pi j [ \phi_i + \omega_i t]  \bigr)  I \bigr]\Bigr)   
\label{e:constantGainQSAsoln}
\end{equation}
\end{proposition}

\Cref{t:constantGainQSAsoln}  illustrates a challenge with fixed gain algorithms:   if we want small steady-state error,  then we require small $\alpha$  (or large $\omega_i$, but this brings other difficulties for computer implementation---\textit{never forget Euler!}).   However, if $\alpha>0$ is very small, then the impact of the initial condition in \eqref{e:ODEstateLinear_soln}  will persist for a long time.    

The  Ruppert-Polyak averaging technique  can be used to improve the steady-state behavior---more on this can be found in \Cref{s:GainQSA} for vanishing-gain algorithms.  It is easy to illustrate the value for the special case considered here.    One form of the   technique is to simply average some fraction of the estimates:
\begin{equation}
\ODEstate^{\text{RP}}_T  
\eqdef
  \frac{1}{T-T_0}  \int_{T_0}^T \ODEstate_t \, dt
\label{e:QSAgenRPpre}
\end{equation} 
For example,  $T_0 = T -  T/5$ means that we average the final 20\%.

\begin{corollary}
\label[corollary]{t:constantGainQSAsolnRP}
Suppose that the assumptions of \Cref{t:constantGainQSAsoln} hold, so in particular $f$ is linear.  
Consider the averaged estimates \eqref{e:QSAgenRPpre}
in which $  T_0= T-T/K $ for fixed $K>1$.   %  Said later:, so that  $ 1/(T-T_0) = K/T$.  
Then,
\[
 \ODEstate^{\text{RP}}_T    =    \theta^\ocp  +    M_T  \theta_0   +B \sum_{i=1}^K     \ODEstate^{\text{RP}\, i}_T 
\]
where
\[
M_T =    \frac{K}{T}   \alpha^{-1} A^{-1} \Bigl[ \exp\bigl(\alpha A T \bigr)  -   \exp\bigl(\alpha A T_0  \bigr)   \Bigr]
\]
and  $ \ODEstate^{\text{RP}\, i}_T  =  \alpha \Upgamma_t^{\text{RP}\, i}  v^i  $ for each $i$ and $t$, with
 $\Upgamma_T^{\text{RP}\, i} $  equal to the integral of  $\Upgamma_t^i $ appearing in 
\eqref{e:constantGainQSAsoln}:
\[
\begin{aligned} 
\Upgamma_T^{\text{RP}\, i}     &=  \frac{K}{T}      \mathrm{Im} \Bigl(  [\alpha A]^{-1}   [\alpha A  -j 2\pi\omega_i  I] ^{-1}  \Bigl[    \exp\bigl( \alpha AT  \bigr)  -     \exp\bigl( \alpha AT_0  \bigr) \Bigr] \Bigr)   
\\
&\quad
	+      \frac{K}{T}   \mathrm{Im}   \Bigl(   \frac{j}{2\pi\omega_i}   [\alpha A  -j 2\pi\omega_i  I] ^{-1}  \Bigl[  
 \exp\bigl( 2\pi [ \phi_i +  \omega_i T ] j \bigr)  -   \exp\bigl( 2\pi [ \phi_i +  \omega_i T _0] j  \bigr)   \Bigr]   \Bigr)
 \end{aligned} 
\]
Hence,  $ \ODEstate^{\text{RP}}_T $ converges to $   \theta^\ocp$ at rate $1/T$.  \qed
\end{corollary}

\subsection{Application to policy iteration}   
\label{s:QSARL}

Consider the nonlinear state space model in continuous time,
\[
\ddt x_t = \Dcs(x_t,u_t) \,, \qquad t\ge 0\,  
\] 
with $x_t\in\Re^n$,  $u_t\in\Re^m$.   Given a cost function $c\colon\Re^{n+m}\to \Re$, our goal is to approximate the optimal value function
\[
J^\oc(x) = \min_\bfmu \int_0^\infty c(x_t, u_t   )\, dt\,, \qquad x=x_0
\]
and approximate the optimal policy.   For this we first explain how policy iteration extends to the continuous time setting.

For any feedback law $u_t = \fee(x_t)$,   denote the associated value function by
\[
J^\fee(x) =  \int_0^\infty c(x_t,\fee(x_t  )   )\, dt\,, \qquad x=x_0. 
\]
This solves a dynamic programming equation:
\sfb{It follows from \Cref{t:ODEconverse} that this }
 \[
 0= c(x,\fee(x)) + \nabla J^\fee\, (x)\cdot \Dcs(x,\fee(x) ) 
\]
The policy improvement step in this continuous time setting defines the new policy as the minimizer:
\[
\fee^+(x) \in \argmin_u \{  c(x, u ) + \nabla J^\fee\, (x)\cdot \Dcs(x,  u )  \}
\]
Consequently, approximating the term in brackets is key to approximating PIA.

An RL algorithm is constructed through the following steps.  First,   add $J^\fee$ to each side of the fixed-policy dynamic programming equation:
\[
 J^\fee\, (x) =    J^\fee\, (x)   + c(x,\fee(x)) + \nabla J^\fee\, (x)\cdot \Dcs(x,\fee(x) ) 
\]
The right-hand side motivates the following definition of the fixed-policy Q-function:   
\[
Q^\fee(x,u) = J^\fee(x) + c(x, u) +\Dcs(x,u )\cdot \nabla J^{\fee}\, (x).
\]
The policy update can be equivalently expressed $\fee^+(x) =\argmin_u Q^\fee(x,u)$,  and 
this Q-function solves the fixed point equation
\begin{align}\label{eq:Q:pi}
Q^\fee(x,u) = \uQ^\fee(x) + c(x, u) +\Dcs(x,u )\cdot \nabla \uQ^\fee\, (x)
\end{align}
where $\uH^\fee(x) = H(x,\fee(x))$ for any function $H$  (note that this is a substitution, rather than the minimization appearing in   Q-learning).

\sfb{\eqref{e:QTD}).  
}

Consider now a family of functions $Q^\theta  $ parameterized by $\theta$, and define the Bellman error for a given parameter as
\begin{align}
\begin{split}
\clE^\theta(x,u)  = -Q^{\theta} (x,u)+ \uQ^{\theta} (x) + c(x,u) +  \Dcs(x,u)\cdot \nabla \uQ^{\theta}\, (x)  
\end{split}
\label{e:BEQ}
\end{align} 
 A model-free representation is obtained, on recognizing that for   any state-input pair $(x_t,u_t)$, 
\begin{equation}
\begin{aligned}
\clE^\theta(x_t,u_t) & = -Q^\theta  (x_t,u_t)+ \uQ^\theta  (x_t)   + c(x_t,u_t)  +  \ddt \uQ^\theta  (x_t)
\end{aligned} 
\label{e:BEonline}
\end{equation} 
The error $\clE^\theta(x_t,u_t) $ can be observed without knowledge of the dynamics $\Dcs$ or even the cost function $c$.   
The goal is to find $\theta^\ocp$ that minimizes the mean square error:
\begin{equation}
\|\clE^\theta\|^2
\eqdef
\lim_{T\to\infty}
\frac{1}{T} \int_0^T  \bigl[
\clE^\theta(x_t,u_t)  \bigr]^2\, dt.
\label{e:ergodicBellman}
\end{equation}
We choose a feedback law with ``exploration'':   \sfb{, of the form introduced in \Cref{s:banditExplore}:}
\begin{equation}
u_t = \feex(x_t,\qsaprobe_t)
\label{s:QSAfeedback}
\end{equation}
chosen so that the resulting state trajectories are bounded for each initial condition, and that the joint process $(\bfmx,\bfmu,\bfqsaprobe)$ admits an ergodic steady state.

Whatever means we use to obtain the minimizer,  this approximation technique defines an approximate version of PIA:   
given a policy $\fee$ and  approximation   $Q^{\theta^\ocp}$, the policy is updated:
\begin{equation}\label{eq.policy.update}
\fee^+(x) =\argmin_u Q^{\theta^\ocp}(x,u)
\end{equation}
This procedure is repeated to obtain a recursive algorithm.
\mindex{Policy iteration!Approximate}

\paragraph{Least squares solution}

Consider for fixed $T$ the loss function
\[
L_T(\theta) = \frac{1}{T} \int_0^T  \bigl[
\clE^\theta(x_t,u_t)  \bigr]^2\, dt
\]
If the function approximation architecture is linear,
\begin{equation}
Q^\theta  (x,u) =  d(x,u) + \theta^\transpose \psi(x,u)\,,\quad \theta\in\Re^d. 
\label{e:Qtheta0}
\end{equation}
in which $d\colon\state\times\ustate\to\Re$.  Then  $L_T$ is a quadratic function of $\theta$:
\[
L_T(\theta) =    \theta^\transpose M \theta - 2 b^\transpose \theta  +  L_T(\Zero) =
 (\theta-\theta^\ocp)^\transpose M  (\theta-\theta^\ocp)  +  L_T(\Zero) 
\]
We leave it to the reader to find   expressions for $M$, $b$, and $L_T(\Zero) $.

In this special case we do not need   gradient descent techniques:  the matrices $M$ and $b$ can be represented as Monte-Carlo,   as surveyed in \Cref{s:QMC}, and then $\theta^\ocp = M^{-1} b$.

\paragraph{Gradient descent}

The first-order condition for optimality is expressed as  a root-finding problem:  
$\nabla_\ODEstate\|\clE^\theta\|^2 = 0$,  and the standard gradient descent algorithm in ODE form is 
\[
\ddt \odestate_t =   - \half a \nabla_\theta  \|\clE^{ \odestate_t } \|^2  = -a    \clE^{ \odestate_t }   \nabla_\theta  \clE^{ \odestate_t } 
\]
with $a>0$.   This is an ODE of the form \eqref{e:ODE_SA},  whose QSA counterpart \eqref{e:QSAgen} is the 
QSA steepest descent algorithm,   
\begin{equation}
\begin{aligned}
\ddt \ODEstate _t &= - a_t \clE^{\ODEstate_t} (x_t,u_t) \zeta^{\ODEstate_t}_t 
\\
\elig_t^\theta & \eqdef \nabla_\theta \clE^\theta (x_t,u_t) 
\end{aligned}
\label{e:Q}
\end{equation}
 Where, based on \eqref{e:BEonline} we can typically swap derivative with respect to time and derivative with respect to $\theta$ to obtain
\[
\nabla_\theta 
\clE^\theta(x_t,u_t)   = -\nabla_\theta Q^\theta  (x_t,u_t)+  \Bigl\{ \nabla_\theta  Q^\theta  (x_t,\fee(x_t) )     +  \ddt \nabla_\theta  Q^\theta  (x_t,\fee(x_t))   \Bigr\}  
\]

\smallskip

The QSA gradient descent algorithm \eqref{e:Q} is best motivated by a nonlinear function approximation, but it is instructive to see how the ODE simplifies for the  
the linearly parameterized family \eqref{e:Qtheta0}.  We have in this case
\[
\elig_t =  - \psi(x_t,u_t) +  \psi( x_t,\fee(x_t) )   + \ddt \psi( x_t,\fee(x_t) )
\]
and $\clE^\theta (x_t,u_t) = b_t+\elig_t^\transpose \theta$ using
\[
b_t  =  c(x_t, u_t) - d(x_t, u_t) + d(x_t, \fee(x_t))   + \ddt d(x_t, \fee(x_t))  
\]
so that \eqref{e:Q} becomes
\begin{equation}
\begin{aligned}
\ddt \ODEstate_t &= - a_t \left[\elig_t \elig_t^\transpose \,\ODEstate_t + b_t \elig_t \right]
\\
\end{aligned}
\label{e:Q:fixed:policy}
\end{equation}

The convergence of~\eqref{e:Q:fixed:policy} may be very slow if the matrix 
\begin{equation}
G \eqdef \lim_{t\to\infty} \frac{1}{t}\int_{0}^t \zeta_\tau\zeta_\tau^\transpose \, d \tau
\label{e:G}
\end{equation}
 has   eigenvalues close to zero  (see \Cref{t:Fast_rho} for a simple explanation).    This can be resolved through the introduction of a larger gain $\bfma$, or a matrix gain.  One approach is to estimate $G$ from data and invert:  
\begin{subequations}
\label{e:fixed:q:matrix}
\begin{align} 
\haG_t &= \frac{1}{t}\int_{0}^t \zeta_\tau\elig_\tau^\transpose \, d \tau,\quad \qquad 0\le t\le T \\
\ddt \ODEstate_t &= - a_t\haG_T^{-1}\left[\elig_t \elig_t^\transpose \,\ODEstate_t + b_t \elig_t \right]  ,\,   \qquad t\ge T 
\label{e:Q:fixed:policy:matrix}
\end{align}
\end{subequations}
This might be motivated by the ODE approximation
\[
\ddt \odestate = - a\{ \odestate  -\theta^\ocp \}
\]
This idea is the motivation for Zap SA and Zap Q-learning \cite{chedevbusmey20b,dev19}.

\begin{figure}[thb]

	\Ebox{.5}{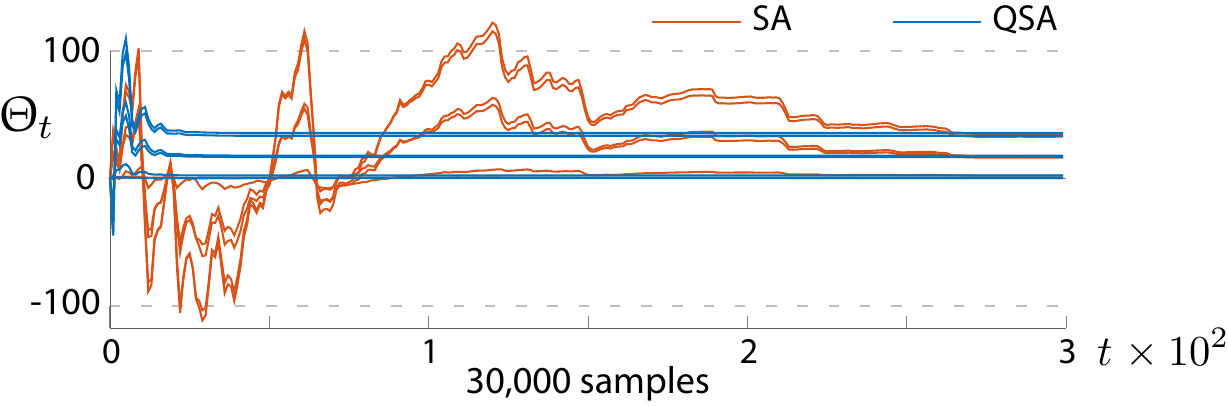}
	
	\caption{   Comparison of QSA and Stochastic Approximation (SA) for policy evaluation.   %Say this in body:  It is observed that QSA converges significantly faster.
		%This is in line with the theoretical results presented in \Cref{s:QSA}. 
	}
	\label{fig.qsa.PI} 
	\vspace{-.5em}
\end{figure}

\paragraph{Numerical example} 

Consider the LQR problem in which $\Dcs(x,u) = Ax+Bu$,  and $c(x,u) = x^\transpose Mx + u^\transpose Ru$, with   $M\ge 0$ and $R>0$.   Given the known structure of the problem, we know that the function $Q^\fee$ associated with any linear policy $\fee(x) = Kx$, takes the form
\[
Q^\fee = \begin{bmatrix} x \\ u  \end{bmatrix}^\transpose \left(
\begin{bmatrix} M & 0\\0 & R  \end{bmatrix}
 + 
 \begin{bmatrix} A^\transpose P + P A  + P   & PB \\
 B^\transpose P & 0  \end{bmatrix}
\right)
\begin{bmatrix} x \\ u  \end{bmatrix}
\]
where $P$ solves the Lyapunov equation 
\[
A^\transpose P + PA + K^\transpose R K + Q = 0
\]
 This motivates a quadratic basis, which for  the special case $n=2$ and $m=1$ becomes 
\[
\{\psi_1,\dots,\psi_6\}
=
\{
x_1^2,x_2^2, x_1 x_2, x_1 u, x_2 u, u^2\}
\]
and there is no harm in setting $d(x,u) \equiv 0$.

%We could omit the $u^2$ term in the parametrization. When we do so, however, we observe numerical instabilities in the algorithm (the explanation is not yet clear). 

In order to implement the  algorithm~\eqref{e:Q:fixed:policy:matrix} we begin with selecting an input of the form
\begin{align}\label{e:offpolicy}
u_t = K_0 x_t + \qsaprobe_t
\end{align}
where $K_0$ is a stabilizing controller and $\qsaprobe_t =\sum_{j=1}^q v^j \sin(\omega_j t + \phi_j)$. Note that $K_0$ need not be the same $K$ whose value function we are trying to evaluate. 

\notes{Say at the start that we are only looking for  a stationary point in general, but here we end up with a simple quadratic objective, no?
\\
Then we run~\eqref{e:Q:fixed:policy:matrix} to obtain $\ODEstate^\star$ which minimizes the Bellman error~\eqref{e:ergodicBellman}. 
}

\begin{wrapfigure}[11]{L}{0.3\textwidth}
\Ebox{1}{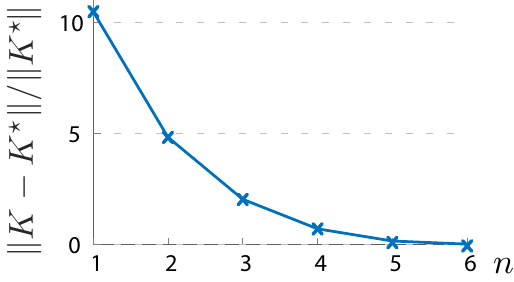}
\caption{Iterations of PIA}
% policy improvement algorithm (PIA)~\eqref{eq.policy.update} where each evaluation is performed by the model-free algorithm~\eqref{e:fixed:q:matrix}.   }
\label{f:PIA_CDC2019} 
\end{wrapfigure}
The numerical results that follow are based on a double integrator with friction:
\[
\ddot y = -0.1 \dot y + u
\]
which can be expressed in state space form using $x = (y,\dot y)^\transpose$:
\begin{equation}\label{e:lin}
\dot x = \begin{bmatrix} 0 & 1\\0 & -0.1 \end{bmatrix} x + \begin{bmatrix}0\\1\end{bmatrix} u  
\end{equation}
We took a relatively large cost on the input: 
\[
 M = I\, \qquad R = 10 
 \]
and gain $a_t = 1/(1+t)$.

\Cref{fig.qsa.PI} shows the evolution of the QSA algorithm for the evaluation of the policy $K=[-1,0]$ using the stabilizing controller $K_0 = [-1,-2]$ and $\qsaprobe$ in~\eqref{e:offpolicy} as the sum of 24 sinusoids with random phase shifts and whose frequency was sampled uniformly between $0$ and $50$ rad/s. The QSA algorithm is compared with the related SA algorithm in which $\qsaprobe$ is ``white noise'' instead of a deterministic signal\footnote{For implementation, both~\eqref{e:fixed:q:matrix} and the linear system~\eqref{e:lin} were approximated using   Euler's method, with time-step of 0.01s.}

  \Cref{f:PIA_CDC2019} shows the weighted error for the feedback gains obtained using the approximate  policy improvement algorithm~\eqref{eq.policy.update} and the optimal controller $K^\star$ (which can be easily computed for an LQR problem). Each policy evaluation was performed by the model-free algorithm~\eqref{e:fixed:q:matrix}.
The PIA algorithm indeed converges to the optimal control gain $K^\star$.
\notes{Revise}

\section{Main Results}
\label{s:GainQSA}

\subsection{Assumptions}

The  following assumptions are imposed throughout the remainder of the paper: 
\begin{romannum}  
 
\item[\textbf{(A1)}] The process $\bfma$ is non-negative, monotonically decreasing, and  
\begin{equation}
\lim_{t\to\infty} a_t = 0,\qquad\int_0^\infty a_r\, dr = \infty.
\label{e:QSA_A5}
\end{equation} 

\item[\textbf{(A2)}]  
The functions $\barf$ and $f$ are Lipschitz continuous:  for   a constant $\Lip_f  <\infty$,
\begin{align*} 
\|\barf(\theta') - \barf(\theta)\| &\le \Lip_f \|\theta' - \theta\|, 
\\
\|f(\theta',\qsaprobe) - f(\theta,\qsaprobe)\| &\le \Lip_f \|\theta' - \theta\|,\quad \theta', \, \theta\in\Re^d\,, \ \qsaprobe\in\prstate 
\end{align*} 
There exists a constant $b_0<\infty$, such that for all $\theta\in\Re^d$, $T>0$,
\begin{equation}
\left\|\frac{1}{T}\int_0^T f(\theta,\qsaprobe_t)\, dt - \barf(\theta)\right\| \le \frac{b_0}{T}(1+\|\theta\|) 
\label{e:pre-ProbeErgodic}
\end{equation}

\item[\textbf{(A3)}] The ODE \eqref{e:ODE_SA} has a globally asymptotically stable equilibrium   $\theta^\ocp$.

\end{romannum}
The process $\bfma$ in (A1) is a continuous time counterpart of the standard step size schedules in stochastic approximation.  
 \notes{maybe give a forward reference, and maybe not: except that we impose monotonicity in place of square integrability.}
 The Lipschitz condition in (A2) is standard,  and  \eqref{e:pre-ProbeErgodic} is
 only slightly stronger than ergodicity of $\bfqsaprobe$ as given by \eqref{e:ergodicA1}. 
 General sufficient conditions on both $f$ and $\bfqsaprobe$ for the ergodic bound \eqref{e:pre-ProbeErgodic}
are given in  \Cref{t:expProbeErgodic}.
Assumptions~(A1) and (A2), along with a strengthening of (A3),  imply convergence of \eqref{e:QSAgen} to $\theta^\ocp$:  see   \Cref{s:QSAconvergence}.

%Sometimes we can identify a ``covariance''  for the scaled error: 
%\begin{equation}
%\barSigma_\theta\eqdef \lim_{T\to\infty}  \frac{1}{T} \int_0^T  \scerror_t\scerror_t^\transpose \, dt.  
%\label{e:cov}
%\end{equation}

\sfb{See covariance discussion commented out}
%Consequently, the error covariance   exists whenever there is a covariance for $\bfqsaprobe^I$: 
%\[
%\barSigma_\theta  =B \Sigma_{\hbox{\footnotesize$\upxi$}^I}  B^\transpose\,,\qquad
%\Sigma_{\hbox{\footnotesize$\upxi$}^I}  \eqdef  \lim_{T\to\infty}  \frac{1}{T} \int_0^T  \qsaprobe^I_t{\qsaprobe^I_t}^\transpose \, dt.  
%\]
%The boundedness assumptions, and the existence of a covariance matrix both hold if the probing signal is a mixture of sinusoids.

\sfb{ This theory is fully developed in \Cref{s:SA_Sigma}, along with a   formula for the covariance matrix $\Sigma_\theta$.}

\medskip

\sfb{REVISE for book!  The challenge in the nonlinear case is ensuring  (A3)---that  \eqref{e:ODE_SA} is globally asymptotically stable---and that the matrix $I + A$ is Hurwitz for the linearization.    
In general this requires modification of the ODE, such as the introduction of   a matrix gain.
One approach discussed in \Cref{s:ZapQSA} 
 is similar to the Newton-Raphson flow, so we expect that this will help to improve transient behavior,  as well as ensure the optimal rate of convergence.
An  approach called Ruppert-Polyak averaging is discussed in
\Cref{s:GainQSA}.  This is far simpler,  but requires (A3).
}

We henceforth assume  that \eqref{e:QSAgen} is convergent, and turn to identification of the convergence rate.   This requires a slight strengthening of (A2):
\begin{romannum}  
\item[\textbf{(A4)}]
 The vector field $\barf$ is   differentiable,  with derivative denoted 
\begin{equation}
A(\theta)  = \partial_\theta \barf \, (\theta)
\label{e:barfDer}
\end{equation}
That is,   $A(\theta)$ is a $d\times d$ matrix for each $\theta\in\Re^d$, with 
$\displaystyle
A_{i,j}(\theta)  = \frac{\partial}{\partial \theta_j}   \barf_i\, (\theta)
$.

Moreover, the derivative  $A$ is  Lipschitz continuous, and      $A^\ocp = A (\theta^\ocp)$ is Hurwitz. 
\end{romannum}  
The matrix-valued function $A$ is uniformly bounded over $\Re^d$, subject to the global Lipschitz assumption on $\barf$ imposed in (A2). 
The Hurwitz assumption implies that  the ODE  \eqref{e:ODE_SA} is exponentially asymptotically stable. 

The final assumption is a substantial strengthening of the ergodic limit \eqref{e:pre-ProbeErgodic}.   For this it is simplest to adopt a ``Markovian'' setting in which the probing signal is itself the state process for a dynamical system: 
\begin{equation}
\ddt\qsaprobe = \qsaDyn(\qsaprobe) 
\label{e:qsaDynamics}
\end{equation}
where $\qsaDyn\colon \prstate \to \prstate $ is continuous,   with $\prstate$   a compact subset of Euclidean space.    
A canonical choice is the $K$-dimensional torus:
$\prstate =\{ x\in \Co^K :   |x_i| =1 \,,\ \ 1\le i\le K\}$,  and $\bfqsaprobe$ defined to allow modeling of excitation as a mixture of sinusoids:   
\begin{equation}
\qsaprobe_t = [\exp(j \omega_1 t) ,\dots,  \exp(j \omega_K t)  ]^\transpose   
\label{e:expProbe}
\end{equation}
with distinct frequencies, ordered for convenience: $0<\omega_1<\omega_2 < \cdots < \omega_K$.       The dynamical system \eqref{e:qsaDynamics}
is linear in this special case.    It is  ergodic, in a sense made precise in \Cref{t:expProbeErgodic}.

\sfb{
Although   deterministic,  the ODE \eqref{e:qsaDynamics}  defines a Markov process on $\prstate$. }

 \Cref{t:expProbeErgodic} also provides justification for the following assumptions for the special case \eqref{e:expProbe}, and under mild smoothness conditions on $f$.   We do \textit{not} assume  \eqref{e:expProbe} in our main results.   
\begin{romannum}
\item[\textbf{(A5)}]  
The probing signal is the solution to \eqref{e:qsaDynamics}, with $\prstate$   a compact subset of Euclidean space. It
has a unique invariant measure $\uppi$ on $\prstate$,   and satisfies the following ergodic theorems for the functions of interest, for each initial condition $\qsaprobe_0\in\prstate$:
\begin{romannum}
\item   For each $\theta$ there exists a solution to Poisson's equation $\haf (\theta,\varble)$ with forcing function $\tilf(\theta,\varble) = f(\theta,\varble) -\barf(\theta)$.   That is,
\begin{equation}
\haf( \theta, \qsaprobe_{t_0} ) = 
\int_{t_0} ^{t_1}   [ f( \theta ,  \qsaprobe_t)  - \barf(\theta) ] \, dt    + \haf( \theta, \qsaprobe_{t_1} )   \,,\qquad 0\le t_0\le t_1
\label{e:PoissonA1}
\end{equation}
with
\[
\barf(\theta)  = \int_\prstate f(\theta,x)\, \uppi(dx)  \qquad \textit{and} \qquad  \Zero =    \int_\prstate \haf(\theta,x)\, \uppi(dx) 
\]

\item  The function $\haf$,   and   derivatives   $\partial_\theta\barf$ and $\partial_\theta f$ 
are   Lipschitz continuous in $\theta$.   In particular,   $\haf$ admits a derivative $\haA$ satisfying
\[
\haA( \theta, \qsaprobe_{t_0} ) = 
\int_{t_0} ^{t_1}   [ A( \theta ,  \qsaprobe_t)  - A(\theta) ] \, dt    + \haA( \theta, \qsaprobe_{t_1} )   \,,\qquad 0\le t_0\le t_1
\]
where $A(\theta,\qsaprobe) = \partial_\theta f(\theta,\qsaprobe)$ and $A(\theta)  = \partial_\theta \barf \, (\theta)$ was defined in \eqref{e:barfDer}.   Lipschitz continuity is assumed uniform with respect to the exploration process: for $\Lip_f  <\infty$,
\begin{align*}  
\| \haf(\theta',\qsaprobe) - \haf(\theta,\qsaprobe)\| &\le \Lip_f \|\theta' - \theta\|
\\
\| A(\theta',\qsaprobe) - A(\theta,\qsaprobe)\| &\le \Lip_f \|\theta' - \theta\| 
\\
\| \haA(\theta',\qsaprobe) - \haA(\theta,\qsaprobe)\| &\le \Lip_f \|\theta' - \theta\| 
,\qquad \theta', \, \theta\in\Re^d\,, \ \qsaprobe\in \prstate
\end{align*}

 \item  Denote $\Upupsilon_t = [  \haA(\theta^\ocp,\qsaprobe_0) - \haA(\theta^\ocp,\qsaprobe_t)    ]    f(\theta^\ocp, \qsaprobe_t) $.   The following limit exists:
\[
\barUpupsilon
\eqdef
\lim_{T\to\infty}\frac{1}{T} \int_0^T  
\Upupsilon_t  \, dt
    =
    -
      \int_\prstate   \haA(\theta^\ocp,x)      f(\theta^\ocp, x)   \,  \uppi(dx) 
\]
and the   partial integrals are bounded:
\[
\sup_{T\ge 0}  \, \Bigl | \int_0^T    [\Upupsilon_t   - \barUpupsilon] \, dt  \Bigr |  <\infty
\]
\qed
\end{romannum} 
\end{romannum} 
Assumption~(A5)~(iii)   is imposed because the vector $\barUpupsilon$ arises in an approximation for the scaled error $\scerror_t$.
\sfb{the first appearance of $\Upupsilon_t  $ is in \Cref{t:YQSA}.}
This assumption is not much stronger than the others.   In particular, the partial integrals will be bounded if there is a bounded solution to Poisson's equation:
\[
\widehat\Upupsilon_{t_0}
 =
\int_{t_0} ^{t_1}   [ \Upupsilon_t   - \barUpupsilon] \, dt    +\widehat\Upupsilon_{t_1}
\]

The notation $\haf$ in \eqref{e:PoissonA1} is used to emphasize the parallels with Markov process and stochastic approximation theory: this  is precisely the solution to Poisson's equation (with forcing function $\tilf (\varble)=  f( \theta ,  \varble)  - \barf(\varble)$) that appears in theory of simulation of Markov processes, average-cost optimal control, and stochastic approximation \cite{glymey96a,asmgly07,metpri87,benmetpri12}.  For the   one-dimensional probing signal defined by the sawtooth function $\qsaprobe_t \eqdef t \mod 1$,   $t\geq 1$, 
a solution to Poisson's equation has a simple form: 
\[
\begin{aligned}
\hag(z) &= -\int_0^z [g(x) - \barg] \, dx  + \hag(0)\,, \qquad z \in [0, 1)
\\
\textit{where}\quad  \barg &= \int_0^1 g(x)\, dx  
\end{aligned}
\]

It will be useful to make the change of notation: 
for $\theta\in\Re^d$ and $T\ge 0$, 
\begin{equation}
 \XiI_T (\theta)   =  \int_0^T [ f( \theta, \qsaprobe_t)  - \barf(\theta) ] \,  dt    =    \haf(\theta,\qsaprobe_0) - \haf(\theta,\qsaprobe_T)
\label{e:NoiseInt}
\end{equation}
where the second equality follows from  \eqref{e:PoissonA1}.  The special case $\theta=\theta^\ocp$ deserves special notation:
\[
  \XiI_T  =  \XiI_T (\theta^\ocp)  =  \int_0^T f( \theta^\ocp, \qsaprobe_t)   \,  dt 
\]

\sfb{where we used the equilibrium condition $ \barf(\theta^\ocp) =\Zero$. }
%May need double integral:    \XiII_T    =  \int_0^T [ \XiI_t  -   \bar\XiI]  \,  dt

\subsection{Main results}

For simplicity we restrict to the vanishing gain algorithm, using $a_t = 1/(1+t)^\rho$,  with $\rho \in (0, 1]$.    \notes{careful in contributions survey in introduction---may need a $g$ in coupling results  ($g$ in front of $\barY$ and...)}

 \begin{theorem}
\label[theorem]{t:Couple_main}
Suppose that (A1)--(A5) hold, and the gain is $a_t = 1/(1+t)^\rho$. 
Then,
\begin{romannum}
\item   $\rho<1$.  Then the following hold:
\begin{equation}
 \begin{aligned}
  \scerror_t   &= \barY  +  \XiI_t   + o(1)  
 \\
\ODEstate_t & =  \theta^\ocp  +   a_t [\barY  +  \XiI_t ] + o(a_t) 
\end{aligned} 
\label{e:QSAcouple_rho}
\end{equation}
where
\begin{equation}
 \barY  =
  (A^*)^{-1} \barUpupsilon
\label{e:barY}
\end{equation}

\item  $\rho = 1$. If $I+A^\ocp$ is Hurwitz, then the convergence rate is $1/t$:
\begin{equation}
 \begin{aligned}
  \scerror_t   &= \barY  +  \XiI_t   + o(1)  
 \\
\ODEstate_t & =  \theta^\ocp  +   a_t [\barY  +  \XiI_t ] + o(a_t) 
\end{aligned} 
\label{e:QSAcouple1}
\end{equation}
\qed
\end{romannum}
\end{theorem}

 \sfb{Applying the arguments used to transform the average \eqref{e:MCQSA} into a simple version of the QSA ODE,  we find that the solution to \eqref{e:QSAgenRPb}  
\\
A standard choice is $a_t = 1 / (1+t)^\rho$,   with $0<\rho <1$.   
 }

The averaging technique of
Ruppert and Polyak   is often presented as  a two time-scale algorithm,   for which a continuous time formulation is defined as follows:
 \begin{subequations}
\begin{align}  
\ddt \ODEstate_t  &=  a_t   f(\ODEstate_t,\qsaprobe_t) \,,    
 \label{e:QSAgenRPa}
  \\
\ddt\ODEstate^{\text{RP}}_t         & = \frac{1}{ 1+t }  [   \ODEstate_t-  \ODEstate^{\text{RP}}_t  ]
\label{e:QSAgenRPb}
\end{align}%
What is crucial in this estimation technique is that the first gain is relatively large:
$\displaystyle
\lim_{t\to\infty}  (1+t) a_t  = \infty
$.
Regardless of $\bfma$, the estimates can be expressed as an approximate average:
\[
\ODEstate^{\text{RP}}_T  =  \frac{1}{1+T} \ODEstate^{\text{RP}}_0  +  \frac{1}{1+T} \int_0^T     \ODEstate_t\,  dt\,,  \qquad T\ge 0
\]
It may not make sense to average over the entire history,  since you may be including wild transients.      Instead, choose a constant $K>1$,    and define
\begin{align}  
\ODEstate^{\text{RP}}_T   =  \frac{1}{T-T_0}  \int_{T_0 }^T    \ODEstate_t\,  dt\,,  \qquad T\ge 0\,, \ \  T_0= T-T/K
\label{e:QSAgenRPc}%
\end{align}%
\label{e:QSAgenRP}%
\end{subequations}%
The choice of $T_0$ is made so that  $ 1/(T-T_0) = K/T$.
The set of equations \eqref{e:QSAgenRP} will be called \textit{Ruppert-Polyak} averaging.

\begin{theorem}
\label{t:Couple_rhoRP}
Suppose that the assumptions of \Cref{t:Couple_main} hold, so that in particular $a_t = g/(1+t)^\rho$ using any $g>0$,  and with $\rho\in(\half, 1)$.   
Then,   with $\ODEstate^{\text{RP}}_T$ defined by either \eqref{e:QSAgenRPb} or  \eqref{e:QSAgenRPc},    we have
\[
\ODEstate^{\text{RP}}_T  = \theta^\ocp + a_T (1-\rho)^{-1}   \barUpupsilon   +    \Upphi_T /T  + o(1/T)
\]
where $\{ \Upphi_t \}$ is a bounded function of time.   Consequently, the convergence rate is $1/T$ if and only if $\barUpupsilon =\Zero$.
\qed
\end{theorem}

\notes{In practice it is often convenient to simply use a fixed constant $a_t\equiv \alpha >0$ in  \eqref{e:QSAgenRPa}.  We can expect some bias in the estimates,   but we will see in examples that the bias is often negligible.   ...  Low bias is suggested by   \Cref{t:constantGainQSAsolnRP}
}

\subsection{Proof outline}

The first step is to explain why we can replace $\theta^\ocp$ with $\barODEstate_t$ in the definition of the scaled error $\scerror_t$.
  \Cref{t:Fast_rho}  provides justification, and makes clear the enormous difference between the choice of $\rho=1$ or $\rho<1$  when using the gain $a_t=1/(1+t)^\rho$:

\begin{proposition}
\label[proposition]{t:Fast_rho} 
Suppose   that (A3) holds,  and that $\barf$ is $C^1$ with $A^\ocp  $   Hurwitz. Fix $\varrho_0>0$ satisfying $\text{Real}(\lambda) < -\varrho_0$ for every eigenvalue $\lambda$ for $A^\ocp$. 
 Then,  there exists $b_0>0$,  $B<\infty$ such that whenever $\|  \odestate_{\SAtime_0} - \theta^\ocp \|\le b_0 $,  the solution   $\{ \odestate_\SAtime : \SAtime\ge \SAtime_0 \}$ of the ODE \eqref{e:ODE_SA} satisfies   
\[
\|
 \odestate_\SAtime  -\theta^\ocp \|  \le  B  \|  \odestate_{\SAtime_0} - \theta^\ocp \| \exp( - \varrho_0 [\SAtime -  \SAtime_0 ] )  \,,\qquad \SAtime \ge \SAtime_0
\]
Consequently, the following   hold for the solution to the ODE \eqref{e:ODE_haSA} using the gain  $a_t = 1/(1+t)^\rho$, with $0\le\rho$,
and $\|  \barODEstate_{t_0} - \theta^\ocp \|\le b_0 $:
\begin{romannum}
\item   If $\rho =1$ then
$
\|  \barODEstate_t  -\theta^\ocp \|  \le  B  \|  \odestate_{\SAtime_0} - \theta^\ocp \| \bigl[ t_0/t ]^{\varrho_0}$ for $ t\ge t_0$. 

\item  Using  any  $0<\rho <1$,
\begin{equation}
\|  \barODEstate_t  -\theta^\ocp \|  \le    \barB_{t_0}   \|  \odestate_{\SAtime_0} - \theta^\ocp \|   \exp \bigl( - \varrho_0(1-\rho)^{-1}    (1+t)^{1-\rho}    \bigr)  \,,\qquad t\ge t_0
\label{e:Fast_rho}
\end{equation}  
where $  \barB_{ t_0}   = B  \exp \bigl(  \varrho_0(1-\rho)^{-1}    (1+t_0)^{1-\rho}    \bigr)$.  
 \qed
\end{romannum} 
\end{proposition}

A significant conclusion is that when $\rho<1$, so that   \Cref{e:Fast_rho} holds, then
\[
\scerror_t = \frac{1}{a_t}  \bigl(  \ODEstate_t -  \theta^\ocp \bigr)   + \epsy^z_t
\]
with $\epsy^z_t =   [    \theta^\ocp-\barODEstate_t   ]/a_t$ vanishing \textit{quickly} as $t\to\infty$.

     \Cref{t:scaled_error_rep} shows how the nonlinear ODE is naturally ``linearized'', provided it is convergent.

\begin{proposition}
\label[proposition]{t:scaled_error_rep}
Suppose that (A1)--(A4) hold, and that  solutions to \eqref{e:QSAgen} converge to $\theta^\ocp$ for each initial condition.  Then, the   scaled error admits the representation 
\begin{equation}
\ddt \scerror_t   =      \bigl[  \dlstep_t   I     + a_t    A (\barODEstate_t)    \bigr]  \scerror_t    +   a_t \Delta_t +    \tilXi_t  \,,\qquad \scerror_{t_0} =0 
\label{e:scaledErrorRep}
\end{equation}
where $\dlstep_t = -   \ddt \log(a_t)     $,   $ \tilXi_t   = f(\ODEstate_t,\qsaprobe_t)     -  \barf ( \ODEstate_t  )  $,  and    $ \|\Delta_t  \| =  
o(  \| \scerror_t  \| )  $  as $t\to\infty$.

  In particular,  setting $A^\ocp = A (\theta^\ocp)$,
\begin{romannum}
\item    With $a_t = g/(1+t)$,
\begin{equation} 
\ddt \scerror_t   =     a_t  \bigl[  g^{-1}  I     +    A^\ocp   \bigr]  \scerror_t    +  a_t \Delta_t   +       \tilXi_t   
\label{e:scaledErrorRep_cor1}
\end{equation}

\item
For any $\rho \in (0, 1)$, using the gain
$a_t = g/(1+t)^\rho$ gives
\begin{equation} 
\ddt \scerror_t   =    a_t        A^\ocp       \scerror_t +  a_t \Delta_t    +       \tilXi_t   
\label{e:scaledErrorRep_cor_rho}
\end{equation}
\end{romannum}
where the definition of $ \Delta_t $ is different in each appearance, but in each case satisfies    $ \|\Delta_t  \| =  
o(  \| \scerror_t  \| )  $.
 \end{proposition}

Note that $\dlstep_t = -\ddt \log(a_t) $ is always  non-negative under (A1).

The challenge in applying \Cref{t:scaled_error_rep} is that the ``noise'' process $\tilXi_t$ appearing in \eqref{e:scaledErrorRep} is non-vanishing, and is not scaled by a vanishing term.    This is resolved through the change of variables:   
denote for $t\ge 0$,
\begin{equation}
Y_t \eqdef  \scerror_t -  \XiI_t(\ODEstate_t)  
\label{e:Y}
\end{equation}
where $\XiI_t(\ODEstate_t) $ is defined in \eqref{e:NoiseInt}. It is shown in \Cref{t:YQSA} that $Y_t$ solves the differential equation
\[
\ddt Y_t   =    
    a_t      \left[A^\ocp Y_t   +  \Delta^Y_t   -  \Upupsilon_t      +    A^\ocp  \XiI_t   \right]
+
 \dlstep_t   [ Y_t   +   \XiI_t   ]
\]
  where  $ \XiI_t  =  \XiI_t (\theta^\ocp) $,  and
   $ \| \Delta^Y_t \|   = o(1 + \| Y_t\|)$ as $t\to\infty$,  and from this we obtain convergence: $\lim_{t\to\infty}  Y_t =\barY$.

%\notes{
%\bl{I've tossed out $\rho=1$ since it is a distraction for ACC 2021}
%}

This leads easily to \Cref{t:Couple_main},  and
\Cref{t:Couple_rhoRP} then  follows---the details can be found in the Appendix, along with proofs of the other results.

%%%---

\section{Gradient-Free Optimization}
\label{s:extremeQSA} 

Consider the unconstrained   minimization problem
\begin{equation}
\min_{\theta \in \Re^d} \Obj(\theta).
\label{e:minObj}
\end{equation}
It is assumed that    $\Obj \colon\Re^d\to\Re$ has a unique minimizer, denoted as $\theta^\ocp$.  The goal here is to estimate $\theta^\ocp$ based on observations of $\Obj(\preODEstate_t)$,  where   $\bfpreODEstate$ is chosen by design.

%The first step is to relax our goal:  find a solution to $\barf(\theta^\ocp) =0$, where 
%\begin{equation}
%\barf(\theta) \eqdef   \nabla \Obj  (\theta) \,,\qquad \theta\in\Re^d\,.
%\label{e:extremeGrad}
%\end{equation}
%This is equivalent to our original objective if   $\Obj$ is convex.  
The   algorithms described below are each based on the following architecture:  construct an ODE of the form
\begin{equation}
\ddt \ODEstate_t  = - a_t \tilnabla_\Obj(t)
\label{e:qSGD}
\end{equation} 
where   $\bfma$ is a non-negative, time-varying gain,   and $ \tilnabla_\Obj(t)$ is designed to approximate the gradient (in an average sense).

The algorithms are similar to both SPSA and ESC (see introduction).  We opt for  \textit{quasi Stochastic Gradient Descent} (qSGD) here,  as algorithms of the form \eqref{e:qSGD} are     cousins of SGD algorithms that are also designed to approximate gradient descent.
Each algorithm is defined based on a $d$-dimensional \textit{probing signal}.   For simplicity, we impose the following normalization conditions, unless stated otherwise:
\begin{subequations}
	\begin{align}  
	\lim_{T \rightarrow \infty }\frac{1}{T} \int_{t = 0}^T \qsaprobe_t \, dt
	&=  0
	\label{eq:zero_mean}
	\\
	\lim_{T \rightarrow \infty }\frac{1}{T} \int_{t = 0}^T \qsaprobe_t \qsaprobe_t^\transpose \, dt &= I
	\label{eq:unit_cov}
	\end{align} 
	\label{e:prob_mean+cov}%
\end{subequations}  This is easily arranged using a mixture of sinusoids.

The first example is a straightforward translation of \eqref{e:one-meas-SPSA}:
\begin{programcode}{qSGD  \#1}%
	For a  given   $d\times d$ positive definite matrix $G$, and   $\ODEstate_0\in\Re^d$,
	\begin{subequations}
		\begin{align}
		\ddt \ODEstate_t &= -  a_t\frac{ 1}{\epsy}G \qsaprobe_t    \Obj( \preODEstate_t )
		\label{e:ESa}
		\\
		\preODEstate_t& =  \ODEstate_t + \epsy \qsaprobe_t
		\label{e:ESb}
		\end{align}
		\label{e:ES}
	\end{subequations}%
\end{programcode}%

The second example falls into the class of ESC algorithms:
\begin{programcode}{qSGD  \#2}%
	For a given   $d\times d$ positive definite matrix $G$, and initial condition $ \ODEstate_0$,
	\begin{subequations}
		\begin{align}
		\ddt \ODEstate_t &= -  a_t \frac{ 1}{\epsy}G    \qsaprobe_t'  \ddt \Obj( \preODEstate_t )
		\label{e:QGDa}
		\\
		\preODEstate_t& =  \ODEstate_t + \epsy \qsaprobe_t
		\end{align}
		\label{e:QGD}
	\end{subequations}%
	where $\qsaprobe_t' =  \ddt \qsaprobe_t $.
	\\
\end{programcode}  

The algorithm qSGD \#1  takes the form \eqref{e:QSAgen}, with
\begin{equation} \label{e:qsgd1}
f(\ODEstate_t,\qsaprobe_t)  = - \frac{ 1}{\epsy} G \qsaprobe_t    \Obj(  \ODEstate_t + \epsy \qsaprobe_t  )
\end{equation}
Under \eqref{e:prob_mean+cov}, a second order Taylor series expansion gives
\begin{align}
\hspace*{-0.5cm}  
\barf_\epsy(\theta) \eqdef \lim_{T \rightarrow \infty }\frac{1}{T} \int_{t = 0}^T f (\theta, \qsaprobe_t) \,  dt
=
-G \nabla \Obj  (\theta) +\Err (\epsy)
\label{e:ergodicQSA_ES1}
\end{align}
where $\| \Err (\epsy)\| = O(\epsy^2)$ if   $\Obj$ is a $C^2$ function.   A similar approximation holds for qSGD~\#2.

The representation \eqref{e:ergodicQSA_ES1}  suggests that the algorithm will approximate gradient descent, provided the right hand side of
\eqref{e:ergodicQSA_ES1} and \eqref{e:qsgd1} is Lipschitz continuous (cf.~Assumption (A2)). However, in many problems  we know that  $\nabla\Obj$ is globally Lipschitz continuous, but $\Obj$ is \textit{not}.  Consider for example a quadratic function.   In this case $f$ is not Lipschitz continuous, contradicting Assumption (A2) of the QSA theory. We can only guarantee global convergence of \eqref{e:ES} if we employ projection of estimates onto a compact region.
%of interest

A slightly more complex algorithm resolves this issue.   This is again a simple extension of one of Spall's SPSA algorithms with two function evaluations.
\begin{programcode}{qSGD  \#3}%
	For a given   $d\times d$ positive definite matrix $G$, and initial condition $ \ODEstate_0$,
	\begin{equation}
	\ddt \ODEstate_t  = -   a_t \frac{ 1}{ 2 \epsy} G \qsaprobe_t  \Bigl\{   \Obj( \ODEstate_t + \epsy \qsaprobe_t  )   -    \Obj( \ODEstate_t - \epsy \qsaprobe_t  )   \Bigr\}
	\label{e:ESlip} 
	\end{equation}%
\end{programcode}%
Denoting by $ a_t f(\ODEstate_t,\qsaprobe_t) $ the right hand side of  \eqref{e:ESlip}, we can show that $f(\theta, \qsaprobe) = -G \qsaprobe \qsaprobe^\transpose \nabla L (\theta) + O(\varepsilon^2)$; it is Lipschitz in its first variable whenever this is true for $\nabla\Obj$.

The vector field $\barf_\epsy$ for  \eqref{e:ESlip} admits the same approximation \eqref{e:ergodicQSA_ES1} under slightly milder conditions on the probing signal with the zero-mean assumption \eqref{eq:zero_mean} dropped.  
%In fact, under mild assumptions, the two vector fields coincide.    We write $\bfqsaprobe
%\eqdist -\bfqsaprobe$ if for any continuous function $g$,
%\[
%\lim_{T\to\infty}
%  \frac{1}{T}\int_0^T \{  g ( \qsaprobe_t )  -  g ( - %\qsaprobe_t )   \}      =0
%\]
%This holds for example if the components of $\bfqsaprobe$ are sinusoids.   
%Under this assumption,  the limit on the left hand side of  \eqref{e:ergodicQSA_ES1} is the same for either algorithm.
Moreover, the following global consistency result can be established:

\begin{proposition}
	\label[proposition]{t:A3}
	Suppose that  the following hold for function and algorithm parameters in  qSGD~\#3:
	\begin{romannum}
		\item Assumption~(A1) holds.
		
		\item The probing signal satisfies \eqref{eq:unit_cov}.
		%\item
		%$L$ has a unique minimizer  $\theta^\ocp \in  \Re^d$.
		
		\item 
		$\nabla L$  is globally   Lipschitz continuous,  and $L$  is  strongly convex, with  unique minimizer $\theta^\ocp \in  \Re^d$.
	\end{romannum}
	Then %there exists $\barepsy>0$ such that 
	for each $\epsy > 0$, there is a unique root $ \theta^\ocp_\epsy$ of $\barf_\epsy$,   satisfying $\| \theta^\ocp_\epsy -  \theta^\ocp \| \le O(\epsy^2)$.  And convergence holds from each initial condition: $ \displaystyle \lim_{t\to\infty}  \ODEstate_t    =  \theta^\ocp_\epsy $.     
	\qed
\end{proposition}

\begin{proof}
	The hypotheses of the proposition imply that Assumptions (A1)-(A2) of the QSA theory hold for $f(\theta, \qsaprobe) = -G \qsaprobe \qsaprobe^\transpose \nabla L (\theta) + O(\varepsilon^2)$ and $\barf_\epsy$ defined in \eqref{e:ergodicQSA_ES1}. Since $L$ is strongly convex, it holds that there is a unique solution to $\nabla L (\theta) = O(\epsy^2)$ for any $\epsy > 0$, so that Assumption (A3) holds as well. Therefore, for each $\epsy > 0$, $\ODEstate_t$ converges to the unique root $ \theta^\ocp_\epsy$ of  $\barf_\epsy$ satisfying $\nabla L (\theta^\ocp_\epsy) = O(\epsy^2)$. Due to strong convexity, we have:
	\begin{align*}
	L(\theta^*) \geq L(\theta^*_\epsy) + (\nabla L(\theta^*_\epsy))^\transpose (\theta^* - \theta^*_\epsy) + \frac{\eta}{2}\|\theta^*_\epsy - \theta^* \|^2
	\end{align*}
	for some $\eta > 0$. Therefore
	\begin{align*}
	\frac{\eta}{2}\|\theta^*_\epsy - \theta^* \|^2 &\leq 
	L(\theta^*) - L(\theta^*_\epsy) + (\nabla L(\theta^*_\epsy))^\transpose ( \theta^*_\epsy - \theta^*)  \\
	& \leq \|\nabla L(\theta^*_\epsy)\| \|\theta^*_\epsy - \theta^* \|
	\end{align*}
	implying that $\|\theta^*_\epsy - \theta^* \| \leq O(\epsy^2)$.
\end{proof}

\begin{figure*}
	\Ebox{1}{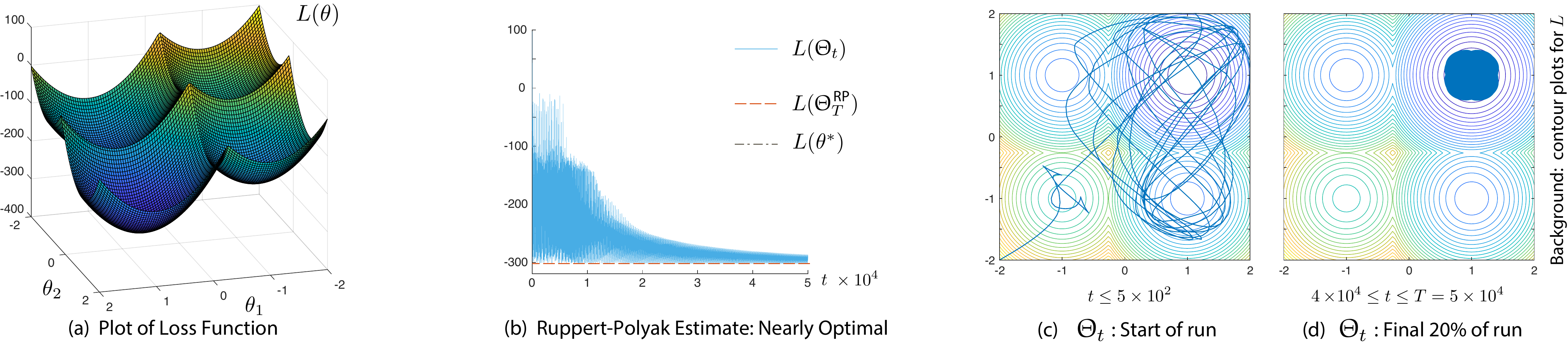}
	\caption{Minimizing a loss function with multiple local minima:  While $\Obj(\ODEstate_t)$ is highly oscillatory,  the estimate $\ODEstate^{\text{RP}}_T$ is nearly optimal.
	}
	\label{f:A4LocalMins_values}
\end{figure*}

\head{qSGD and simulated annealing}

The primary motivation for SPSA, ESC, and qSGD algorithms  is that they can be run based purely on observations of the loss function.   A secondary benefit is that the probing can be designed to emulate a simulated annealing algorithm.   The example described below was designed to illustrate this point.

The first plot on the left in  \Cref{f:A4LocalMins_values} is a highly non-convex function  $\Obj$, defined as the
``soft min'' of convex quadratic functions:
\[
\Obj(\theta) = - \log\Bigl( \sum_{k=1}^4 \exp\bigl(-  \{ z_i +  \| \theta - \theta^i \|^2 \} /\sigma^2 \bigr) \Bigr)
\]
with  $\sigma=1/10$,   and
\[
\textstyle
\{ [ \theta^i , z_i  ] \} =
\bigl\{   \bigl[  \binom{-1}{-1}, -1 \bigr]
\,,
\  
    \bigl[  \binom{-1}{1}, -2 \bigr]
\,,
\  
    \bigl[  \binom{1}{-1}, -2 \bigr]
\,,
\  
    \bigl[  \binom{1}{1}, -3 \bigr]    \bigr\} 
\]
The minimizer   is     $\theta^\ocp \approx  \binom{1}{1}$,  with $L(\theta^\ocp)\approx -300$.

The gain process $ a_t = \min \{  \bara,   (1+t)^{-\rho} \}$ was used
with  $\rho =0.9$ and $\bara=10^{-3}$. The probing signal was chosen to satisfy \eqref{e:prob_mean+cov}:
\[
\qsaprobe_t =  \sqrt{2}   [  \sin(  t\omega_1) ,      \sin(  t\omega_2) ]^\transpose   
\]
with $\omega_1=1/4$ and $\omega_2= 1/e^2$ chosen   to obtain attractive plots---higher frequencies lead to faster convergence.   The value $\epsy=0.15$ was chosen for the scaling (recall \eqref{e:ESb}).
These meta-parameters were obtained by trial and error: if $\bara$ or $\epsy$ is too small, then we are   trapped in a local minima.

The ODE was approximated using a standard Euler scheme with 1~sec sampling interval: the crude ODE approximation led to the requirement that $\omega_1/\omega_2$ is irrational.

The two plots on the right hand side in
\Cref{f:A4LocalMins_values}
show the evolution of $\ODEstate_t$ in $\Re^2$ for $0\le t \le 5\times 10^4$,  with $\ODEstate_0 = (-2,-2)^\transpose$.   
Plots (b) and (d) indicate that the estimates exhibit significant variation throughout the run, but in plot (d) it is clear that they are trapped within the region of attraction of the global minimum.   

What's more, averaging is highly successful: the estimate
$\ODEstate^{\text{RP}}_T$ was   obtained as  the average of $\ODEstate_t$ over the final 20\%\ of the run.    
It is found that $\Obj(\ODEstate^{\text{RP}}_T)$ is only a small fraction of one percent greater than $\Obj(\theta^\ocp)$.

\section{Quasi Policy Gradient Algorithms}
\label{s:ActorCriticQSA}

A simple example is illustrated in \Cref{f:MountainCar}, in which the two dimensional state space is   position and velocity:
\[
x^k\in \state = [z^{\text{min}},  z^{\text{goal}}   ]\times [-\barv, \barv]  
\]
where $z^{\text{min}}$ is a lower limit for position,   and the target position is $z^{\text{goal}}$.  The velocity is bounded in magnitude by $\barv>0$.
The input $u$ is the throttle position (which is negative when the car is in reverse).  
The special case adopted in   \cite[Ch.~10]{sutbar18} is modeled in discrete time:
\begin{equation}  
\begin{aligned}
z_{k+1}  &=   \llbracket  x_1^k + x_2^k     \rrbracket_1
\\
v_{k+1}          & =  \llbracket  v_{k}   + 10^{-3}  u_k  -  2.5\times10^{-3}  \cos(3 z_{k})  \rrbracket_2
\end{aligned} 
\end{equation}
with  $z^{\text{min}} = -1.2$,  \ $ z^{\text{goal}} = 0.5$, \  and \  $\barv = 7\times 10^{-2}$.
The brackets are projecting the values of $z_{k+1} $ to the interval $ [z^{\text{min}},  z^{\text{goal}}   ] = [-1.2, 0.5]$,
and   $v_{k+1} $ to the interval $ [-\barv, \barv]  $.    

\begin{wrapfigure}[8]{R}{0.25\textwidth}
	
	\vspace{-.5em}
	\centering 
	\includegraphics[width=1\hsize]{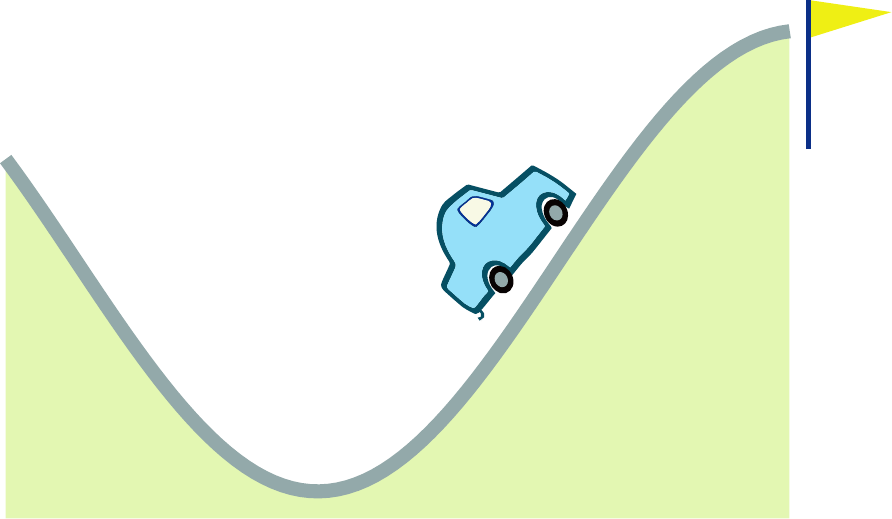}
	\caption{Mountain Car}
	\label{f:MountainCar}
\end{wrapfigure}

Due to state and input constraints,  a feasible policy will sometimes put the car in reverse, and travel at maximal speed away from the goal to reach a higher elevation to the left.  Several cycles back and forth may be required to reach the goal.

The control objective is to reach the goal in minimal time, but this can also be cast  as a total cost optimal control problem.   Let $x^e = ( z^{\text{goal}} , 0)^\transpose  $,   and reduce the state space so that $x^e$ is the only state $x=  ( z,v)^\transpose \in\state$  satisfying $z = z^{\text{goal}}$.  
Let $c(x,u) = 1$ for all $x,u$ with $x\neq  x^e $,   and $c( x^e ,u) \equiv 0$.

For $\theta\in\Re$, consider the policy  
\begin{equation} 
u_k = \phi_\theta(x^k) = 
\begin{cases}
1   \quad   &  \textit{if} \ \ z_{k} + v_k  \le \theta
\\
\text{sign}(v_k )  \qquad  & \textit{else} \end{cases} 
\label{e:Aggressive_theta}
\end{equation}  
The policy accelerates the car towards the goal whenever the estimate   $\haz_{k+1} = z_k + v_k$ is at or below the threshold $\theta $.

\Cref{f:ThreePlotsPolicyMC_theta_mp8+0}  shows trajectories of position as a function of time from three initial conditions, and with two instances of this policy:  $\theta=-0.8$,  and $\theta=-0.2$.    The former is a much better choice from initial condition $z(0) =-0.6$:
we see that the time to reach the goal is nearly twice as long when using $\theta=-0.2$ as compared to  $\theta=-0.8$.

Application of qSGD to estimate $\theta^\ocp$ is a form of policy gradient (PG) reinforcement learning.

\begin{figure}[h]
	
	\Ebox{1.0}{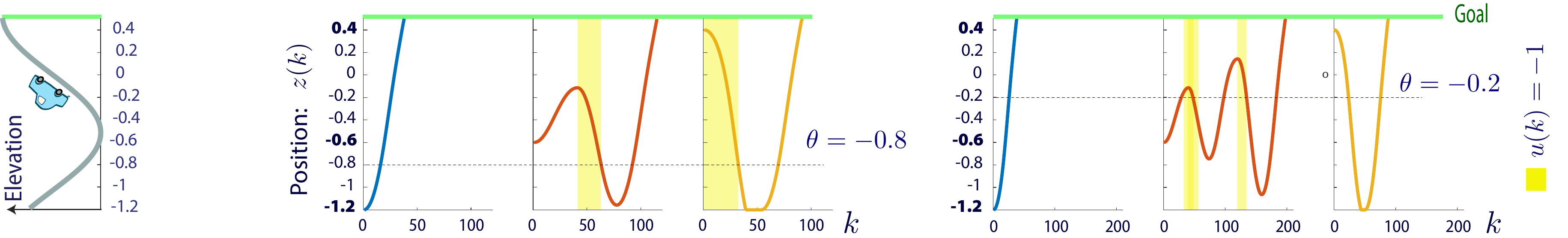}
	\caption{Trajectories for the Mountain Car for two policies (differentiated by $\theta$), and three initial positions.}
	\label{f:ThreePlotsPolicyMC_theta_mp8+0}
\end{figure}

The total cost in this example coincides with the time to reach the goal.   For fixed initial condition $x^0\in\Re^d$,  we might apply qSGD to estimate the minimizer of the corresponding total cost $J_\theta(x^0)$ over $\theta$.   To avoid infinite values, we consider the   loss function   $\Obj(\theta ) = \min\{ J^{\text{max}}, J_\theta(x^0) \}$; 
the value $J^{\text{max}} = 5\times 10^3$ was used in these experiments.   

A discrete-time counterpart of qSGD  \#1\  is 
%(the recursion  \eqref{e:ES})
\begin{subequations}
	\begin{align}
	\ODEstate_{n+1} &= \ODEstate_n + \alpha_{n+1}  \frac{1}{\epsy} G \qsaprobe_{n+1}   \Obj (\preODEstate_{n+1} ) 
	\\
	\preODEstate_{n+1}           & =  \ODEstate_n + \epsy \qsaprobe_{n+1}
	\label{e:ESdiscrete_b}%
	\end{align}%
	\label{e:ESdiscrete}%
\end{subequations}%
The algorithm is \textit{episodic} in the sense that the observation of $\Obj( \preODEstate_{n+1} ) $ is obtained only when the car reaches the goal, or the time limit $ J^{\text{max}}$ is reached.  

It may be more valuable to introduce randomization in the initial condition.  In this case we introduce a second probing signal $\{ \qsaprobe^x_n \}$.   We proceed as above, but define
\begin{equation}
L( \preODEstate_{n+1}) = J_\theta(x^{n+1}) \text{ with  }  x^{n+1} = \llbracket   x^0 + \epsy^x \qsaprobe^x_{n+1}  \rrbracket
\label{e:QuasiRandomIC}
\end{equation}
where the brackets again indicate projection onto the state space $\state$.  The goal then is to minimize the \textit{average cost}:  
%with $N\gg1$  a large integer,
\begin{equation}
\Obj(\theta) =
\Expect[J_\theta(X) ] =  \lim_{N\to\infty}  \frac{1}{N} \sum_{n=1}^N  
\min\{J^{\text{max}},  J_\theta(\qsaprobe^x_n)  \} 
\label{e:ACMC_J}
\end{equation}

The signal $\{\qsaprobe^x_n =  (\qsaprobe^z_n ,\qsaprobe^v_n )^\transpose \}$ was chosen to cover the state space uniformly:     introduce two signals that are quasi-uniform and independent on $[0,1]$:
\notes{refer to a lemma on why these are uniform and independent}
\[
\clW_n^v =  \text{frac}(n r_v )  \,,\qquad 
\clW_n^z =  \text{frac}(n r_z )  \,, 
\]
where  ``frac'' denotes the fractional part of a real number,   $r_v, r_z$ are irrational, and their ratio is also irrational.  
Then define 
\begin{equation}
\begin{aligned}
\qsaprobe^v_n  &=   \barv (2 \clW_n^v   -1)  
\\
\qsaprobe^z_n  & =   z^{\text{min}}  + [  z^{\text{goal}}  -z^{\text{min}}  ] \clW_n^z 
\end{aligned} 
\label{e:qsaprobeX}
\end{equation} 
The values $r_v = \pi$ and $r_z = e$ were chosen in these    experiments. Also,  $ x^0 = 0$ and $ \epsy^x =0$ in \eqref{e:QuasiRandomIC},  giving  $x_{n+1} =  \qsaprobe^x_{n+1}   $.

\begin{figure}[h]
	\Ebox{.65}{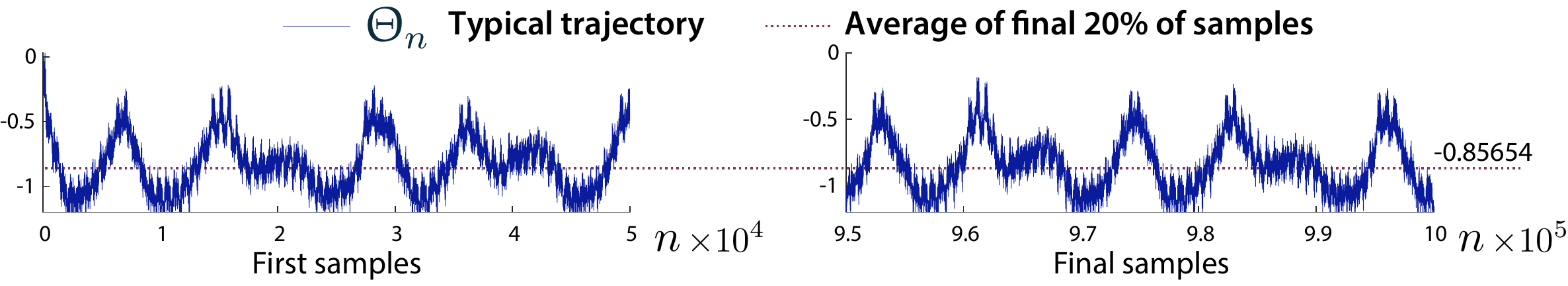}
	\caption[qPG~\#1 via Gradient-Free Optimization]{qSGD~\#1    for the Mountain Car using  the gradient-free optimization algorithm \eqref{e:ESdiscrete} using a large constant step-size.
	} 
	\label{f:QSA_theta_epsy} 
	\vspace{-.5em}
\end{figure}

A run using the episodic algorithm
(\ref{e:ESdiscrete}, \ref{e:QuasiRandomIC}) is shown in 
\Cref{f:QSA_theta_epsy}, with constant step-size $\alpha_n = 0.1$,    $\epsy = 0.05$,  and $ \qsaprobe_n =\sin(    n) $.  The large step-size was chosen simply to illustrate the exotic nonlinear dynamics that emerge from this algorithm.   It would seem that the algorithm has failed, since the estimates oscillate between $-1.2 $ and $-0.3$ in steady-state, while the actual optimizer is   $\theta^\ocp\approx -0.8$.   The dashed line shows the average of $\{\ODEstate_n \}$ over the final 20\%\ of estimates.   This average is very nearly optimal, since the objective function is nearly flat for $\theta$ near the optimizer.

\begin{figure*}[h]
	\centering
	\includegraphics[width=\hsize]{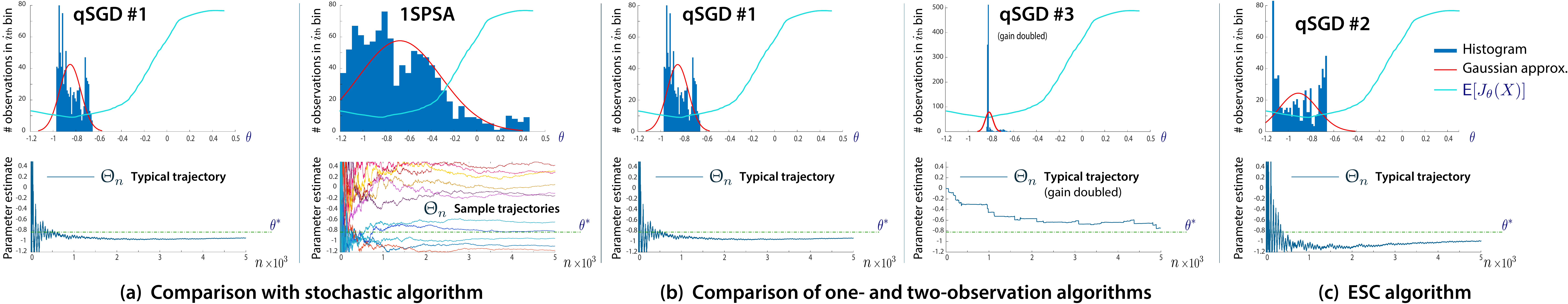}
	\caption{PG for Mountain Car:   (a)  QSA and traditional randomized exploration,
		(b) qSGD \#1 and \#3,
		(c)
		using   qSGD \#2:    \Cref{e:ES2}
	}  
	\label{f:QSA+MC-ALL} 
	\vspace{-.5em}
\end{figure*}

\Cref{f:QSA+MC-ALL}~(a)  shows results from   $10^3$ independent runs, each with horizon length $\Hor = 10^4$.   In each case, the parameter estimates evolve according to \eqref{e:ESdiscrete},  to obtain estimates $\{ \ODEstate_n^i : 1\le n\le \Hor \,, \ 1\le i\le 10^3\}$.  The two columns are distinguished by  the probing signals $\qsaprobe_n$ and $\qsaprobe^x_n $.  For QSA the probing signal $\bfqsaprobe$ was a sinusoid,  with phase $\phi$ selected independently in the interval $[0,1)$,  in each of the $10^3$ runs.
The probing sequence $\bfqsaprobe^x$ was fixed as \eqref{e:qsaprobeX}.

\notes{old caption:   The first column shows results using quasi gradient ascent using QSA, and the second using i.i.d.\ exploration (similar to the Kiefer Wolfowitz approach).   The final estimates vary significantly on this short run, with time horizon $\Hor = 5\times 10^3$.   In the case of QSA,  the final estimates are always found in an interval near $\theta^\ocp$,  and on this interval $\Expect[J_\theta(X) ]$ is nearly constant.    The estimates obtained using i.i.d.\ exploration show far more variability,   with many extreme outliers for which $\theta_\Hor$ is near the maximum of $\Expect[J_\theta(X) ]$,  with time horizon doubled to  $\Hor =  10^4$.     The plots shown in the lower right hand corner are selected trajectories for illustration, including several of the worst outliers.  }

The results displayed in the second column  of \Cref{f:QSA+MC-ALL}~(a)  used an independent sequence for the probing signals,  each uniform on their respective ranges  (in particular,  the distribution of $\qsaprobe_n$ was chosen uniform on the interval $[-1,1]$ for each $n$).    The label 1SPSA refers to the algorithm of Spall based on i.i.d.\ exploration \cite{spa97,spa03} (recall \eqref{e:one-meas-SPSA}).

\sfb{ (see the \textit{Notes} section at the end of this chapter for history).
}

It was found that qSGD  \#3\ results in far more reliable estimates:  see comparison in \Cref{f:QSA+MC-ALL}~(b).    However, this method will not be reliable in the presence of system disturbances, since it requires on two observations of the objective function.

Algorithm qSGD  \#2\  is also easily adapted to this application:
\begin{subequations}
	\begin{align}
	\ODEstate_{n+1} &= \ODEstate_n + \alpha_{n+1}  \frac{1}{\epsy} G  \qsaprobe_{n+1} '  \,   \Obj (\preODEstate_{n+1} )'
	\\
	\preODEstate_{n+1}           & =  \ODEstate_n + \epsy \qsaprobe_{n+1}
	\end{align}%
	\label{e:ES2}%
\end{subequations}%
where the primes denote approximations of the derivatives appearing in \eqref{e:QGDa}: 
\[
\qsaprobe_{n+1} '  = \frac{1}{\delta} \bigl(   \qsaprobe_{n+1} - \qsaprobe_n \bigr) \,, 
\quad
\Obj (\preODEstate_{n+1} )'   =  \frac{1}{\delta} \bigl(   \Obj (\preODEstate_{n+1} ) - \Obj (\preODEstate_{n} ) \bigr)
\]
with  $\delta>0$ the sampling interval.

%  \begin{wrapfigure}[12]{l}{0.25\textwidth}
%
%  \vspace{-.5em}
%  \centering 
%     \includegraphics[width=1\hsize]{QSA_Hist_Hor1e4_method2_Journal.pdf}
%     \caption{qPG for Mountain Car using        \Cref{e:ES2}}  
%\label{f:QSA_Hist_Hor1e4_method2} 
%\end{wrapfigure}

A histogram and sample path of parameter estimates are shown in \Cref{f:QSA+MC-ALL}~(c) , based on algorithm \eqref{e:ES2}   with $\delta^2 = 0.5$,  and  all of the same choices for parameters, except that the step-size was reduced to avoid large initial transients:
$
\alpha_n = \min(1/n^{0.75} ,
0.05 )
$.
This results in $ \alpha_n =  1/n^{0.75}  $ for $n\ge 55$.

Based on the histogram, the performance appears slightly worse than  observed for method~\#1 in \Cref{f:QSA+MC-ALL}~(a) ,  but these outcomes are a product of particular choices for algorithm parameters.   

\notes{Figure commented out:
	\Cref{f:QSA_Hist_Hor1e4_method2_bigProbe}:
	$
	\epsy_n =  \alpha_n^{0.75}
	$.}

% 
%   \begin{wrapfigure}[15]{r}{0.4\textwidth}
% 
%
%     \includegraphics[width=1\hsize]{QSA_Hist_Hor1e4_method2_bigProbe.pdf}
%          \caption{Method 2,  $\epsy_n = \alpha_n^{0.75}$} 
%
%\label[figure]{f:QSA_Hist_Hor1e4_method2_bigProbe} 
%\end{wrapfigure}

\notes{see commented text for parameters}
%
%if STEPflag==1
%    
%    stepSizeExp= .75;
%    stepSizeGain= 1;  
%    
%    alpha_max=0.05; 
%
%    probeExp=  0;  %0.55;      
%    ProbestepSizeGain=1;
%    
%    
%    alpha_probe_max=0.05; 
%
%    
%    %These values work great for method 1:
%    % probeExp=  0.75;      
%    % ProbestepSizeGain=1;
%
%    
%        
%    % Gain Stepsize sequence
%    alpha = 1./[1:Hor];
%    alpha =stepSizeGain* alpha.^stepSizeExp;
%    
%    alpha=min(alpha,alpha_max);
%    
%    alpha_probe = 1./[1:Hor];
%    alpha_probe =ProbestepSizeGain* alpha_probe.^probeExp;
%    
%    alpha_probe=min(alpha_probe,alpha_probe_max);
%
%    %Try this:
%    alpha_probe = alpha.^0.75;

\section{Conclusions}
\label{s:conc}

It is evident that QSA is worthy of far greater attention in terms of both theory and applications.   The case for applications to optimization and optimal control was made in  \Cref{s:extremeQSA,s:ActorCriticQSA}, where the most obvious examples are deterministic control systems, such as in robotics. 
Theoretically speaking, we have shown under general conditions that the rate of convergence is much faster for QSA as opposed to SA. 
If there is one take-home message from these results, it is this:  \emph{do not inject randomness in your algorithm}.

There are several open questions for future research:
\\
$\bullet$ 
Theory in the quasi Monte-Carlo literature may shed doubt on the $1/t$  convergence rate established in \Cref{t:Couple_main,t:Couple_rhoRP},   since it is well known that the optimal rate is of the form $[\log(t)]^d/t$,  with $d$   the dimension  (Section 3 of Chapter~9 of \cite{asmgly07} contains an accessible survey).    It may be that sharper results obtained in the present paper are a result of the smoothness conditions imposed on the functions involved. 
Another possibility is that the theory of QMC is posed in discrete time,  and that  an Euler approximation will destroy the $1/t$ rate.   If so, this   motivates   a more accurate ODE approximation.      \notes{WIP:  momentum etc here}
\\
$\bullet$   The convergence-rate obtained in \Cref{t:Couple_rhoRP} for Ruppert-Polyak averaging is currently a mystery, because the vector $\barUpupsilon$ defined in (A5) is not well understood.  
A current topic of research concerns conditions under which this is zero, so we obtain the optimal $1/t$ convergence rate.

More exciting questions concern the creation of more efficient schemes to optimize the convergence rate.   The crisp approximation obtained in    \Cref{t:Couple_main} will likely lead to new approaches.
\\
$\bullet$ 
\textit{What about high dimensions?}
The qSGD algorithms are easy to code, and quickly converge to an approximately optimal parameter for the examples considered in this paper.    In high dimensions we can't expect to blindly apply any of these algorithms.    If $\theta\in\Re^d$ with large $d$,  then the complexity of the probing signal must grow as well.
For example,  consider the choice of probing signal \eqref{e:SawSinusoidalQSA},   where $i$ ranges from $1$ to $d=K=1000$.   If the frequencies $\{\omega_i\}$ are chosen in a narrow range,  then the limit \eqref{eq:unit_cov} will converge very slowly.     The rate will be faster if the frequencies are widely separated, but we then need a much higher resolution Euler approximation to implement an algorithm.     

This challenge is well understood in the optimization literature.    One approach to create a reliable algorithm is to employ \textit{block coordinate descent} to effectively reduce the dimension of the optimization problem. This requires two ingredients:
\begin{romannum}
\item   A sequence of timepoints $T_0 = 0<T_1 < T_2< \cdots$
\item  A sequence of ``parameter blocks''   $B_k \subset \{1,\dots,d \}$ for each $k\ge 0$,  where the number of elements $d_B$ in $B_i$ is far smaller than $d$. 
\end{romannum}
The qSGD ODE \eqref{e:qSGD} is modified so that   $ \ODEstate_t(i)$ is held constant on the interval  $[T_k, T_{k+1})$ for $i\not\in B_k$,    and
\[
\ddt \ODEstate_t(i  )  = - a_t [ \tilnabla_\Obj(t) ]_i\,\qquad i \in B_k\,, \ \ t\in [T_k, T_{k+1})
\]
The alternating direction method of multipliers (ADMM)  employs a similar scheme.

% 
%   \begin{wrapfigure}[15]{r}{0.4\textwidth}
% 
%
%     \includegraphics[width=1\hsize]{QSA_Hist_Hor1e4_method2_bigProbe.pdf}
%          \caption{Method 2,  $\epsy_n = \alpha_n^{0.75}$} 
%
%\label[figure]{f:QSA_Hist_Hor1e4_method2_bigProbe} 
%\end{wrapfigure}

 \notes{see commented text for parameters}
%
%if STEPflag==1
%    
%    stepSizeExp= .75;
%    stepSizeGain= 1;  
%    
%    alpha_max=0.05; 
%
%    probeExp=  0;  %0.55;      
%    ProbestepSizeGain=1;
%    
%    
%    alpha_probe_max=0.05; 
%
%    
%    %These values work great for method 1:
%    % probeExp=  0.75;      
%    % ProbestepSizeGain=1;
%
%    
%        
%    % Gain Stepsize sequence
%    alpha = 1./[1:Hor];
%    alpha =stepSizeGain* alpha.^stepSizeExp;
%    
%    alpha=min(alpha,alpha_max);
%    
%    alpha_probe = 1./[1:Hor];
%    alpha_probe =ProbestepSizeGain* alpha_probe.^probeExp;
%    
%    alpha_probe=min(alpha_probe,alpha_probe_max);
%
%    %Try this:
%    alpha_probe = alpha.^0.75;

%\section{Start of a literature survey}
%
%
% 
%
%
%   
% \sfb{The use of abstract ODE models to verify stability of stochastic recursions also appears in queueing networks \cite{dai95a,daimey95a,CTCN} and MCMC  \cite{formeymoupri06a}.
%\\
%We will revisit this approach to stability verification in \Cref{s:SA}.
%\\
%The ODE  \eqref{e:ZAPnonlinearSA-ODE} was introduced in the economics literature, which led to the comprehensive analysis  by Smale~\cite{sma76}   for   smooth $\barf $.
%The term \textit{Newton-Raphson flow} for \eqref{e:ZAPnonlinearSA-ODE} was introduced in the deterministic control literature \cite{shibucwarseaege18,warseaegebuc17}.  
%The Zap SA algorithm   was introduced at the same time, and based on the same ODE~\cite{devmey17a}.
%}

\addcontentsline{toc}{section}{References}

\bibliographystyle{abbrv}
\bibliography{strings,markov,q,QSA,bandits} 

\clearpage
\appendix

\centerline{\LARGE \bf Appendices}

\bigskip

This Appendix is organized into three sections:  some of the convergence theory in 
\Cref{s:QSAconvergence} is taken from   \cite{berchecoldalmehmey19a,berchecoldalmehmey19b} (duplicated here for convenience).   
The material supporting the proof of \Cref{t:convergenceBM_QSA}  is all new.
    \Cref{s:quasiMarkov} contains theory to justify Assumption~(A5),  and \Cref{s:tech} contains   proofs of the main results related to rates of convergence for QSA.

\section{Convergence of QSA}   % \eqref{e:QSAgen}}
\label{s:QSAconvergence}

Analysis of QSA is based on consideration of the associated  ODE \eqref{e:ODE_SA} in which the ``averaged'' vector field is introduced in \eqref{e:ergodicA1}: 
\begin{equation}
\barf(\theta) = \lim_{T\rightarrow\infty}\frac{1}{T}\int_0^T f(\theta,\qsaprobe_t)\, dt,\ \ \textrm{for all }\theta\in\Re^d.
\label{e:ergodicQSA}
\end{equation}
The first step in the theory is to find assumptions  to ensure that the limit  exists.   Following this,  stability conditions are found for  the averaged ODE  \eqref{e:ODE_SA} that ensure that ODE \eqref{e:ODE_SA} and the  QSA ODE \eqref{e:QSAgen} converge to the same limit.

 \notes{\rd{Now very little is from CDC 2019!   Note that there is a complete proof of the Borkar Meyn result}  }

%The extension of stability and convergence results from the classical stochastic model 
%\eqref{eq:SA} to the deterministi®c analog \eqref{e:QSAgen} requires some specialized analysis since the standard methods are not directly applicable.  In particular,  the first step in \cite{bor20a} and other references is to 
%write \eqref{eq:SA} in the form,
%\[
%\theta_{n+1} = \theta_n +\alpha_n \bigl(\barh(\theta_n) + M_n\bigr) \, ,
%\]
%where $\bfmM$ is a martingale difference sequence (or a perturbation of such a sequence).  This is possible when $\bfmW$ is i.i.d.,  or for certain Markov $\bfmW$. %  \cite{mamakshw90}.    
%A similar transformation is not possible for any class of deterministic $\bfqsaprobe$. 

The starting point in an ODE approximation is a temporal transformation:
substitute in \eqref{e:QSAgen} the new time variable given by
\begin{equation}
	\SAtime= g_t \eqdef  \int_0^t a_r\, dr,\qquad t\ge 0.
\label{e:SAtime}
\end{equation} 
The time-scaled process is then defined by  
\begin{equation}
\haODEstate_{\SAtime} \eqdef \ODEstate(g^{-1}(\SAtime)) = \ODEstate_ t   \Big| _{t = g^{-1}(\SAtime) }
\label{e:varscaled}
\end{equation}  
For example, if $a_r=(1+r)^{-1}$, then 
\begin{equation}
\SAtime=\log(1+t) \ \
	\text{and} \ \ 
\qsaprobe(g^{-1}(\SAtime)) = \qsaprobe(e^\SAtime-1).
\label{e:a-r-inv}
\end{equation}

The chain rule of differentiation gives 
\begin{equation*}
\frac{d}{d\SAtime}\ODEstate(g^{-1}(\SAtime)) = f(\ODEstate(g^{-1}(\SAtime)),\qsaprobe(g^{-1}(\SAtime)).
\end{equation*}  
That is,  the time-scaled process solves the ODE,
\begin{equation}
\label{e:QSAgenscaled}
\frac{d}{d\SAtime}\haODEstate_{\SAtime} = f(\haODEstate_{\SAtime},\qsaprobe(g^{-1}(\SAtime)).
\end{equation}
The two processes $\bfODEstate$ and $\bfhaODEstate$ differ only in time scale, and hence, proving convergence of one  proves that of the other. For the remainder of this section we will deal exclusively with $\bfhaODEstate$;  it is on the `right' time scale for comparison with  $\bfodestate$, the solution of \eqref{e:ODE_SA}.

The first step in establishing convergence of QSA is to show that the solutions are bounded in time. Two approaches can be borrowed from the literature:  Lyapunov function techniques,   or the ODE at $\infty$ introduced in \cite{bormey00a,bor20a}.

\sfb{previous sections
\Cref{s:BM_theorem}.
}

When applying the techniques of   \cite{bormey00a,bor20a} we require the vector field at $\infty$:
\sfb{\Cref{s:BM_theorem} }
\begin{equation}
\barf_\infty (\theta) 
\eqdef
\lim_{r \to \infty }  r^{-1} \barf ( r \theta)  \,,\qquad \theta\in\Re^d 
\label{e:barfinfty_QSA}
\end{equation} 
See \cite{imaginaryPaper20} for a proof of the following:   \notes{Mind if I cite my imaginary book?}

\begin{theorem}
\label{t:convergenceBM_QSA} 
Suppose that Assumptions (A1)--(A3) hold, along with the following two conditions:
\begin{romannum}
\item     The limit \eqref{e:barfinfty_QSA}
exists for all $\theta$ to define a continuous function $\barf_\infty \colon\Re^d\to\Re^d$
\item
The origin is  globally asymptotically stable for the  ODE at $\infty$:
\begin{equation}
\ddt \odestate_t^\infty = \barf_\infty ( \odestate_t^\infty) \,,\qquad \theta\in\Re^d 
\label{e:BMscaling}
\end{equation}
\end{romannum} 
Then  the solution to \eqref{e:QSAgen} converges to $\theta^\ocp$ for each initial condition. 
\end{theorem}

When applying Lyapunov function techniques we impose the following:
\begin{romannum}

\item[\textbf{(QSV1)}] There exists a continuous function $V:\Re^d\rightarrow\Re_+$ and   constants  $c_0>0$,  $\delta_0>0$ such that, for any initial condition $\odestate_0$ of
\eqref{e:ODE_SA}, and
any $0\le T\le1$,
the following bounds hold whenever
$\|\odestate_s\|>c_0$,
\[
V(\odestate_{s+T}) - V(\odestate_s) \le -  \delta_0 \int_0^T  \|\odestate_t\| \, dt
\]
The Lyapunov function is Lipschitz continuous:  
there exists a constant $\Lip_V  <\infty$ such that  $
\|V(\theta') - V(\theta)\| \le \Lip_V  \|\theta' - \theta\|
$ for all $\theta$,  $\theta'$.

\end{romannum}
Assumption (QSV1) ensures that there is a Lyapunov function $V$ with a strictly negative drift whenever $\bfodestate$ escapes a ball of radius $c_0$. This   is used to establish boundedness of   $\bfODEstate$. 
If $V$ is differentiable then we have
\[
\ddt V(\odestate_t)   \le -    \|\odestate_t\| \, ,  \quad \textit{whenever $\|\odestate_t\|>c_0$}
\]
The integral form is chosen since sometimes it is easier to establish a bound in this form.   In particular, the proof of \Cref{t:convergenceBM_QSA} 
below is based on the construction of a solution to (QSV1).

\paragraph{Verifying (QSV1) for a linear system.}  
Consider the ODE \eqref{e:ODE_SA} in which $\barf(x) = Ax$ with $A$ a Hurwitz $d\times d$ matrix.   There is a quadratic function $V_2(x) =   x^\transpose P x$ with $P\in \Re^{d\times d}$ satisfying the Lyapunov equation $PA +A^\transpose P = -I$, with $P>0$.  
 Consequently, solutions to \eqref{e:ODE_SA}  satisfy
 \[
 \ddt V_2(\odestate_t)  = -\| \odestate_t\|^2
 \]
Choose $V= \sqrt{V_2}$, so that by the chain rule
 \[
 \ddt V(\odestate_t)  =     
 - \frac{1}{2}  \frac{1}{\sqrt{V_2(\odestate_t) } } \| \odestate_t\|^2   
 \] 
This $V$ is a Lipschitz solution to (QSV1), for any $c_0>0$.
\qed

%The QSA algorithm is consistent under these assumptions:

\begin{theorem}
\label{t:convergenceV_QSA} 
Under Assumptions (A1)--(A3) and (QSV1),  the solution to \eqref{e:QSAgen} converges to $\theta^\ocp$ for each initial condition. 
\end{theorem}

\subsection{ODE Solidarity}

Here we establish \Cref{t:convergenceBM_QSA,t:convergenceV_QSA} under the \textit{assumption} that solutions to \eqref{e:QSAgen} are ultimately bounded in the following sense:
there exists $b < \infty$ such that for each $\theta\in\Re^d$,  $z\in\prstate$,  there is a $T_\theta$ such that
\begin{equation}
	\|\haODEstate_{\SAtime}\|\le b \textrm{\ \ for all } \SAtime\ge T_\theta\,,  \quad \textit{when} \ 
    \  \haODEstate_0 = \theta \, , \ \qsaprobe_0 = z
\label{e:ODEub}
\end{equation}
 Verification of this stability condition is contained in  \cite{berchecoldalmehmey19a}, under the assumptions of \Cref{t:convergenceV_QSA}.
A complete proof under the assumptions of 
\Cref{t:convergenceBM_QSA} 
is contained in \Cref{s:bm}.

Define $\odestate^\SAtime(w)$, $w\ge \SAtime $, to be the unique solution to \eqref{e:ODE_SA} `starting' at $\haODEstate_{\SAtime}$:
\begin{equation}
\label{eq:init}
\ddw\odestate^\SAtime(w) = \barf(\odestate^\SAtime(w)),\, \, w\ge \SAtime ,  \, \, \odestate^\SAtime_{\SAtime} = \haODEstate_{\SAtime}.
\end{equation}
The following result is required in the proof of either stability result.  The proof of \Cref{thm:limsuplimsup} can be found in  \cite{berchecoldalmehmey19a}, and is similar to results in the SA literature (e.g.,  Lemma~1 in Chapter~2 of \cite{bor08a}).  \notes{update with bor20a}

\begin{lemma}
\label[lemma]{thm:limsuplimsup}
Assume that $\bfhaODEstate$ is bounded. 
Then, for any $T>0$,  
\begin{align*}
\lim_{\SAtime \to\infty} &
  \sup_{v\in[0,T]}\Bigl\|
		\int_\SAtime^{\SAtime +v} \bigl[ f(\haODEstate_w,\qsaprobe(g^{-1}(w)) - \barf(\haODEstate_w ) \bigr]\, dw
							\Bigr \| = 0
\\
\lim_{\SAtime \to\infty} &
 \sup_{v\in[0,T]}\|\haODEstate_{\SAtime+v} - \odestate^\SAtime(\SAtime+v)\|  = 0.      \
\end{align*}  
\end{lemma}

\begin{proposition}[Boundedness Implies Convergence]
\label[proposition]{t:bddConv}
Suppose that (A1)--(A3) hold,  and that the ultimate boundedness assumption \eqref{e:ODEub} holds.   Then,
 the solution to \eqref{e:QSAgen} converges to $\theta^\ocp$ for each initial condition.

\end{proposition}

\begin{proof}
 Under the assumptions of the proposition, there exists $b<\infty$ such that for any $\theta\in\Re^d$, there is a time $T_\theta$ such that
$\|\odestate^\SAtime_{\SAtime}\| = \|\haODEstate_{\SAtime}\|\le b$,   
for  $ \SAtime \ge T_\theta$.

 By the definition of global asymptotic stability, for every $\epsy>0$, there exists a $\Hor_\epsy>0$, independent of the value $\odestate^\SAtime_{\SAtime}$, such that 
\[
	\|\odestate^\SAtime(\SAtime+v) - \theta^\ocp\| < \epsy  \quad \textrm{ for all }v\ge \Hor_\epsy \,,  \quad \textit{whenever $\|\odestate^\SAtime_{\SAtime}\|   \le b$}
\]
  \Cref{thm:limsuplimsup} gives,
\begin{align*}
	\limsup_{ \SAtime \rightarrow\infty}\|\haODEstate_{\SAtime+\Hor_\epsy }   - \theta^\ocp\| 
	&\le \limsup_{\SAtime \rightarrow\infty}\|\haODEstate_{\SAtime+\Hor_\epsy }-\odestate^\SAtime_{\SAtime+\Hor_\epsy }\| 
+ \limsup_{\SAtime \rightarrow\infty}\|\odestate^\SAtime_{\SAtime+\Hor_\epsy } - \theta^\ocp\| 
 \le \epsy. 
\end{align*}
Since $\epsy$ is arbitrary, we have the desired limit.
% \[
% 	\lim_{\SAtime \rightarrow\infty} \|\haODEstate_{\SAtime} - \theta^\ocp\|  =\lim_{\SAtime \rightarrow\infty} \|\haODEstate_{\SAtime+\Hor_\epsy } - \theta^\ocp\|  = 0.
% \] 
\end{proof}
 
\subsection{ODE@${\infty}$}
\label{s:bm}
 
The remainder of this section contains technical results concerning  the proof of \Cref{t:convergenceBM_QSA}.
It may not be surprising that the following result makes several appearances in the proofs. 
\begin{proposition}[Grönwall Inequality]   
\label[proposition]{t:B-GI}
Let $\bfalpha$, $\bfbeta$ and $\bfmz$ be real-valued functions defined on an interval $[0,\Hor]$, with $\Hor>0$. 
 Assume that $\bfbeta$ and $\bfmz$ are continuous.  
 %   This is automatic, if the others are continuous:,  and that the negative part of  $\bfalpha $ is integrable on this interval.
\begin{romannum}
\item If $\bfbeta$ is non-negative and if $\bfmz$ satisfies the integral inequality
\begin{equation}
  z_t\leq \alpha _t+\int_0^t\beta_s z_s\, ds  
\label{e:GIi}
\end{equation}
Then Grönwall Inequality holds:
\begin{equation}
z_t\leq \alpha _t+\int_0^t\alpha_s\beta_s\exp \Bigl(\int_s^t\beta_r\,\, dr\Bigr)\, ds,\qquad 0\le t\le \Hor.
\label{e:GI}
\end{equation}
 
\item If, in addition, the function $\bfalpha$ is non-decreasing, then
\begin{equation}
z_t\leq \alpha _t\exp \Bigl(\int_0^t\beta_s\,\, ds\Bigr),\qquad 0\le t\le \Hor.
\label{e:GIii}
\end{equation}
\qed
\end{romannum}
\end{proposition}

Grönwall's Inequality implies  
 a crude bound that is needed in approximations:

\begin{proposition}
\label[proposition]{t:ODEbdd}
Consider the ODE \eqref{e:ODE_SA}, subject to the  Lipschitz  condition in (A2).      Then, 
\begin{romannum}
\item
There is a constant $B_f$ depending only on $\Dcs$ such that
\[
\| \odestate_t  \| 
\le
\bigl( B_f + 
 \| \odestate_0  \|  \bigr)  e^{{\Lip_f} t}   - B_f \,,\qquad   t\ge 0
\]
\item 
If there is an equilibrium $ \theta^\ocp$, then for each initial condition,
\[
\| \odestate_t  -  \theta^\ocp \| 
\le
\| \odestate_0 -  \theta^\ocp \| e^{{\Lip_f} t}  \,,\qquad   t\ge 0
\]
\qed
\end{romannum} 
\end{proposition}

The main step in the proof of \Cref{t:convergenceBM_QSA} is to show that the assumptions of the theorem imply that the ODE \eqref{e:ODE_SA} is ultimately bounded.
We first need to better understand the special properties of the solution to \eqref{e:BMscaling}:

\begin{lemma}
\label[lemma]{t:BMscaling}  
Suppose that (A2) holds, and the limit \eqref{e:barfinfty_QSA}  exists for all $\theta$ to define a  continuous function $\barf_\infty \colon\Re^d\to\Re^d$.  Suppose moreover that
the origin is    asymptotically stable for \eqref{e:BMscaling}.     Then the following hold:
\begin{romannum}
\item  For each $\theta\in\Re^d$ and $s\ge 0$,
\[
\barf_\infty (s \theta) =s\barf_\infty (\theta)   
\]

\item   If $\{  \odestate_t^\infty  :t\ge 0  \} $ is any solution to the ODE \eqref{e:BMscaling},   and $s>0$,   then $\{ y_t = s\odestate_t^\infty : t\ge 0 \} $ is also a solution, starting from $y_0 = s\odestate_0^\infty \in\Re^d$.

\item   The origin is \emph{globally} asymptotically stable for \eqref{e:BMscaling}, and convergence to the origin is exponentially fast:
for some $R<\infty$ and $\rho>0$,
\[
\|  \odestate_t^\infty \|  \le  R  e^{-\rho t}  \|  \odestate_0^\infty \|  \,, \qquad  \odestate_0^\infty \in\Re^d
\]
\end{romannum}
\end{lemma}

\begin{proof} 
Consider first the scaling result in part (i):    from the definition \eqref{e:barfinfty_QSA},   with $s>0$,  
\[
\barf_\infty (s \theta) = s    \lim_{r \to \infty }(s r)^{-1}\barf( sr\theta)   =  s\barf_\infty (\theta)
\]
The case $s=0$ is trivial, since it is clear that $\barf_\infty (\Zero) =0$.  This establishes (i).

Next, write
\[
\odestate_t^\infty = \odestate_0^\infty + \int_0^t \barf_\infty ( \odestate_\tau^\infty) \, d\tau 
\]
Multiplying both sides by $s$ and applying (i) gives (ii).

Under the assumption that the origin is asymptotically stable,  there exists $\epsy>0$ such that 
$\lim_{t\to\infty}  \odestate_t = \Zero$,
whenever $\| \odestate_0 \|\le \epsy$.   Moreover, the convergence is uniform:    there exists $T_0>0$ such that 
\[
\| \odestate_{T_0} \| \le \half \epsy\qquad\textit{whenever    $\| \odestate_0 \|\le \epsy$}
\]
Next,  apply scaling:   for any initial condition $ \odestate_0$,  
consider $y_t = s\odestate_t^\infty $   using $s = \epsy/\| \odestate_0 \|$, chosen so that $ \|y_0 \| =\epsy$.
  Then $\|y_{T_0}\|  \le \half \epsy  = \half \|y_0 \|$,  implying  
\[
\| \odestate_{T_0} \| \le \half   \| \odestate_0\|   \qquad \odestate_0 \in\Re^d
\]
This easily implies (iii) by iteration, as follows:
for any $t$ we can write $t= n T_0 + t_0$,  with $0\le t_0 < T_0$, so that  
\[
\| \odestate_{t} \| \le \half   \| \odestate_{(n-1) T_0 + t_0} \|   \le 2^{-n}   \| \odestate_{ t_0} \|    
\]   
   Grönwall's Inequality gives $  \| \odestate_{ t_0} \|    \le e^{\Lip_f}  \| \odestate_{0} \|   $, so that
   \sfb{   \Cref{t:ODEbdd}
 in current draft}
 \[
\| \odestate_{t} \| \le  2    e^{\Lip_f}\,  2^{-(n+1) }   \| \odestate_{0} \|    
\]
where the right hand side has been arranged to make use of the bound $t\le (n+1)T_0$, giving  $2^{-(n+1) } \le \exp( -\log(2) t/T_0 )$.  We arrive at the bound in   (iii) with $R= 2 e^{\Lip_f} $ and $\rho = \log(2) /T_0 $.
\end{proof}

\begin{lemma}
\label[lemma]{t:BMscalingSolidarity}  
Under the assumptions of \Cref{t:BMscaling},  for each  $T<\infty$ and $\epsy\in (0,1] $,   there exists  $K_T<\infty $ independent of $\epsy$,  and   $B_T(\epsy) <\infty$ such that for all solutions to \Cref{e:ODE_SA,e:BMscaling}
from common initial condition $\odestate_0$,
\[ 
  \| \odestate_t-\odestate_t^\infty \| \le   B_T(\epsy) +   K_T [1+ \| \odestate_0  \| ]  \epsy 
\]
\end{lemma}

\begin{proof} 
Denote
\[
\clE(\theta) = \| \barf(\theta) - \barf_\infty(\theta) \|
\]
so that by \Cref{t:BMscaling},  with $s=\|\theta\|$,
\[
s^{-1} \clE(\theta) =    \| \barf_s(\theta/s) - \barf_\infty(\theta/s) \|  
\]
Because the functions $\{\barf_s :  s\ge 1\}$ are uniformly Lipschitz continous,  the right hand side converges to zero uniformly in $\theta\neq \Zero$.   Consequently,
\[
\clE(\theta) =   o(\| \theta \|) 
\]
Let's think about what this means:  for any $\epsy>0$,  there exists $N(\epsy)<\infty$, such that $\clE(\theta) \le \epsy \| \theta \| $ whenever
   $\| \theta \| \ge N(\epsy)$.   From this we get the   simpler looking bound:
 \begin{equation}
\clE(\theta)   \le   B_\epsy +  \epsy \| \theta \| \,,   \qquad\textit{where} \ \  B_\epsy= \max  \{ \clE(\theta) :  \| \theta \| \le N(\epsy) \}
\label{e:BMerror1}
\end{equation}

For any initial condition $\odestate_0$ we compare the two solutions:
\[
\begin{aligned}
\odestate_t &= \odestate_0 + \int_0^t \barf ( \odestate_\tau) \, d\tau    
   \\
\odestate_t^\infty &= \odestate_0 + \int_0^t \barf_\infty ( \odestate_\tau^\infty) \, d\tau 
\end{aligned} 
\]
Write $z_t = \| \odestate_t-\odestate_t^\infty \|$ and use the preceding definition to obtain,
\[
\begin{aligned}
z_t &\le  \int_0^t  \|  \barf_\infty  ( \odestate_\tau)  -  \barf_\infty ( \odestate_\tau^\infty)  \| \, d\tau 
+ \int_0^t   \clE(\odestate_\tau)  \, d\tau
\\
  &\le   {\Lip_f}  \int_0^t  \|   \odestate_\tau -  \odestate_\tau^\infty  \| \, d\tau 
+ \int_0^t   \clE(\odestate_\tau)  \, d\tau
\end{aligned} 
\]

Grönwall's Inequality in its second form \eqref{e:GIii} holds, with $\beta_t \equiv {\Lip_f}$,   and $\alpha_t$ the second integral, giving
\[
z_t  \le   e^{{\Lip_f} t} \int_0^t   \clE(\odestate_\tau)  \, d\tau
		 \le   e^{{\Lip_f} t} \int_0^t  \bigl\{ B_\epsy + \epsy  \| \odestate_\tau\| \} \, d\tau
\]
where the second inequality uses \eqref{e:BMerror1}, with $\epsy>0 $ to be chosen.    
\Cref{t:ODEbdd} gives $ \| \odestate_\tau\| \le  \bigl\{ B_f + 
 \| \odestate_0  \|  \bigr\}  e^{{\Lip_f} \tau}$, so that 
\[
z_t = \| \odestate_t-\odestate_t^\infty \| \le t e^{{\Lip_f} t} B_\epsy   +  \epsy e^{{\Lip_f} t}   \bigl\{ B_f + 
 \| \odestate_0  \|  \bigr\} \bigl\{  {\Lip_f}^{-1} e^{{\Lip_f} t} \bigr\} 
\] 
\end{proof}

\Cref{t:LyapBM_QSA}   combined with  \Cref{t:convergenceV_QSA}   imply
\Cref{t:convergenceBM_QSA}.

\begin{proposition}
\label[proposition]{t:LyapBM_QSA} 
Under the assumptions of  
\Cref{t:convergenceBM_QSA},  there is a Lipschitz continuous function $V$ that saisfies (QSV1).   
\end{proposition}

\begin{proof}
Choose $T>0$ so that $ \| \odestate_t^\infty \|  \le \half  \| \theta \| $ when $t\ge T$,  for any solution to \Cref{e:BMscaling}, from any initial condition $\odestate_0^\infty =\theta$.
We then define 
\[
\begin{aligned}
V^\infty(\theta) & = \int_0^T   \| \odestate_t^\infty \|  \,  dt\,,\qquad   && \odestate_0^\infty =\theta
   \\
V(\theta) & = \int_0^T V( \odestate_t )   \,  dt\,,\qquad   && \odestate_0 =\theta
\end{aligned} 
\]
The Grönwall Inequality implies that each is a Lipschitz continuous function of $\theta$.  Moreover, applying \Cref{t:BMscaling} it follows that  the first is radially homogeneous:   $V^\infty(s\theta) = sV^\infty(\theta)$ for each $\theta$ and $s>0$,   and satisfies the lower bound for some $\delta>0$:
\[
V^\infty(\theta)\ge \delta \|\theta\| 
\]
Consequently, this is a Lyapunov function for the ODE@$\infty$:
 for each initial condition $\odestate_0^\infty =\theta $,  
\[
V^\infty(\odestate_T^\infty) = \int_0^T   \| \odestate_{t+T}^\infty \|  \,  dt  \le \half V^\infty(\theta)   \le   V^\infty(\theta) - \half \delta \| \theta\|
\]

The next step is to show that  a similar bound holds with $\odestate_T^\infty$ replaced by $\odestate_T$.  Let $\Lip_V$ denote the Lipschitz constant for $V^\infty$.    The bound above combined with
\Cref{t:BMscalingSolidarity}  gives
\[
V^\infty(\odestate_T)    \le  V^\infty(\theta) - \half \delta \| \theta\|   +  \Lip_V \Bigl(    B_T(\epsy) +   K_T [1+ \| \theta  \| ]  \epsy   \Bigr)
\]
Fix $\epsy\in (0,1)$ so that 
\[
\Lip_V    K_T   \epsy    \le   \delta/4
\]
giving
\[
V^\infty(\odestate_T)    \le  V^\infty(\theta) - \fourth \delta \| \theta\|   + K_V^\infty
\]
with $K_V^\infty =  \Lip_V  (    B_T(\epsy) +   K_T     )$.

To complete the proof, write
\[
V(\odestate_{s})   = \int_s^T V^\infty ( \odestate_t )   \,  dt    +    \int_0^s V^\infty ( \odestate_{T+t})   \,  dt\
\]
The preceding bound gives
\[
V^\infty ( \odestate_{T+t})      \le V^\infty ( \odestate_{t})  - \fourth \delta \|  \odestate_{t}\|   + K_V^\infty
\]
so that
\[
V(\odestate_{s})  \le V(\odestate_0)        - \fourth \delta    \int_0^s \|  \odestate_{t}\| \,  dt  +    s  K_V^\infty
\]
This bound is essentially equivalent to (QSV1).   \notes{I got lazy here!}
\end{proof}

\section{Deterministic Markovian model}
 \label{s:quasiMarkov}

Here we explain how to verify  assumption (A5),  and obtain representations for the solution to Poisson's equation.   To  simplify notation we fix $\theta$ and $1\le i\le d$,   denote $g(\qsaprobe_t)\eqdef f_i (\theta, \qsaprobe_t)$,   and consider the associated Poisson equation \eqref{e:PoissonA1} in the new notation:
\notes{To do:  forcing function does not need normalization -- correct in body}
\begin{equation}
\hag( \qsaprobe_{t_0} ) = 
\int_{t_0} ^{t_1}   [ g( \qsaprobe_t)  - \barg] \, dt    + \hag( \qsaprobe_{t_1} )  \,,\qquad 0\le t_0\le t_1
\label{e:TSergodicFishSoln}
\end{equation} 
Recall that $g$ is known as the forcing function,   $\barg$ its steady-state mean,  and $\hag$ is the solution (known as the relative value function in some applications).

Assumption (A5) is analogous to common assumptions in the study of simulation or stochastic approximation algorithms when $\bfqsaprobe$ is a Markov process \cite{glymey96a,benmetpri90}.   Conditions for a well behaved solution to Poisson's equation are available, subject to conditions on the Markov process and the function.  In particular, for stochastic differential equations (SDEs), a non-degeneracy condition known as hypoellipticity is a first step,  and then a solution to Poisson's equation exists subject to a   Lyapunov function drift condition \cite{glymey96a}.     

While the process $\bfqsaprobe$ defined by \eqref{e:qsaDynamics} is Markovian, it is purely degenerate in the sense that Poisson's equation \eqref{e:TSergodicFishSoln}
 in differential form is a first order PDE:
\[
g (z) + 
\partial  \hag (z)  \cdot \qsaDyn(z)   =  \barg\,,\qquad z\in \prstate
\]  
To the best of our knowledge, a general theory for solutions is not presently available.     \Cref{t:expProbeErgodic} concerns the
 special case \eqref{e:expProbe} for which $\qsaDyn$ is a linear function of $z\in  \Co^K$.

\begin{lemma}
\label[lemma]{t:expProbeErgodic}
Suppose that $g\colon \Co^K\to \Re$ admits the Taylor series representation,
\begin{equation}
g(z)   = \sum_{n_1,\dots,n_K} a_{n_1,\dots,n_K}   z_1^{n_1}  \cdots z_K^{n_K} \,,\qquad z\in\prstate\,, 
\label{e:TSergodic}
\end{equation}
where the sum is over all $K$-length sequences in $\nat^K$, and the  coefficients  $\{ a_{n_1,\dots,n_K}  \} \subset \Co^K$ are absolutely summable:
\begin{equation}
 \sum_{n_1,\dots,n_K}   |a_{n_1,\dots,n_K}  |<\infty
\label{e:FourierAbs}
\end{equation} 
Then,  with  $\bfqsaprobe$ defined in \eqref{e:expProbe},
\begin{romannum}
\item 
 The ergodic limit holds:
\[
\barg = 
\lim_{T\to\infty}\frac{1}{T} \int_0^T  g( \qsaprobe_t) \, dt  =     \int_0^1\cdots \int_0^1  g\bigl( e^{2\pi j t_1}  ,\dots , e^{2\pi j t_K}   \bigr) \, dt_1\cdots dt_K  
\]
where  $\barg = a_0$  (the coefficient when   $n_i=0$ for each $i$).

\item   There exists a solution $\hag\colon \Co^K\to \Re$  to \eqref{e:TSergodicFishSoln}.  
It is of the form \eqref{e:TSergodic}:
\begin{equation}
\hag(x)   = \sum_{n_1,\dots,n_K} \hat{a}_{n_1,\dots,n_K}   x_1^{n_1}  \cdots x_K^{n_K} 
\label{e:TSergodicFish}
\end{equation}
in which $|\hat{a}_{n_1,\dots,n_K}  | \le |{a}_{n_1,\dots,n_K}  |/\omega_1$ for each  coefficient.
\end{romannum}
\end{lemma}

\begin{proof}
Complex exponentials and the Fourier representation are used to obtain the simple formula:
\[
g( \qsaprobe_t)=   \sum_{n_1,\dots,n_K} a_{n_1,\dots,n_K} 
		  \exp\bigl(  \{ n_1\omega_1 +\cdots+ n_K \omega_K \} jt \bigr)
\]
The absolute-summability assumption \eqref{e:FourierAbs} justifies Fubini's Theorem:
\[
\int_{t_0}^{t_1}  [ g( \qsaprobe_t)  - \barg]  \, dt  =   \sum_{n_1,\dots,n_K} a_{n_1,\dots,n_K} \int_{t_0}^{t_1} 
		  \exp\bigl(  \{ n_1\omega_1 +\cdots+ n_K \omega_K \} jt \bigr)\, dt  =  \hag( \qsaprobe_{t_0} )  - \hag( \qsaprobe_{t_1} ) 
\]
where $\hag$ is given by \eqref{e:TSergodicFish} with $  \hat{a}_{0} =0 $   (that is,  $n_k=0$ for each $k$),  and for all other coefficients
 \[
  \hat{a}_{n_1,\dots,n_K} =   {a}_{n_1,\dots,n_K}  \{ n_1\omega_1 +\cdots+ n_K \omega_K \}^{-1} j
 \]
\end{proof}

The lemma then justifies (A5) provided $f_i(\theta,\varble)$ satisfies the Taylor series bound for each $i$ and $\theta$,  along with the derivatives $\frac{\partial}{\partial\theta_j}   f_i(\theta,\varble)$, for each $i,j$.  While an explicit formula for $\haf$ is not required in any algorithm,    bounds may be valuable in finer convergence rate analysis of QSA algorithms. 
In particular,    the approximation of the scaled error $\scerror_t\eqdef \frac{1}{a_t}(\ODEstate_t-\barODEstate_t)$ obtained in  \Cref{t:Couple_main} depends on $\haf(\theta^\ocp, \qsaprobe_t)$.

%\notes{Poor reader trying to understand this!  Where is the Taylor series?  Where is the complex number?  
%\\
%It then follows from  \Cref{t:expProbeErgodic} that 
%}

\section{Technical Proofs}
\label{s:tech}

The solution to \eqref{e:ODE_SA} is related to the solution to \eqref{e:ODE_haSA} through a temporal transformation:  the proof of \Cref{t:barODEstate} is an application of the chain rule.
\begin{lemma}
	\label[lemma]{t:barODEstate}
	Let $\{ \odestate_\SAtime : \SAtime\ge \SAtime_0 \}$ denote the solution to \eqref{e:ODE_SA} with   $\odestate_{\SAtime_0} =   \ODEstate_{t_0}$,   and $\SAtime_0 = g_{t_0}$ \emph{(with time-change defined in \eqref{e:SAtime}).}     The solution to \eqref{e:ODE_haSA} is then given by
	\[
	\barODEstate_t  = \odestate_\SAtime \,,   \quad t\ge t_0\, ,  \qquad  \textit{with }\quad \SAtime= g_t =  \int_0^t a_r\, dr
	\]
	\qed
\end{lemma}

\begin{proof}[Proof of \Cref{t:scaled_error_rep}]
   Taking derivatives of each side of \eqref{e:scaled_error} gives, by the product rule,
\[
\begin{aligned} 
\ddt \scerror_t 
&=
  \ddt 
 \Bigl(
 \frac{1}{a_t}  \bigl(  \ODEstate_t - \barODEstate_t \bigr) 
  \Bigr)
  \\
 &=  \Bigl(
  - \frac{1}{a_t ^2}  \ddt a_t   \Bigr)     \bigl(  \ODEstate_t - \barODEstate_t \bigr)   +    f(\ODEstate_t,\qsaprobe_t)     -  \barf  ( \barODEstate_t  )   
    \\ 
  &= % -\ddt \log(a_t)  
  		r_t   \scerror_t     +    f(\ODEstate_t,\qsaprobe_t)     -  \barf  ( \barODEstate_t  )   
\end{aligned} 
\]
where in the final equation we used
the chain rule for the derivative of a logarithm
(recall that  $\dlstep_t = -   \ddt \log(a_t)     $), and
 the definition of $\scerror_t$.
    
   On adding and subtracting   $  \barf  ( \ODEstate_t  )  $,
we arrive at  a suggestive decomposition:
\[
\begin{aligned} 
\ddt \scerror_t   &=     \dlstep_t   \scerror_t     +  \underbrace{   \bigl[ \barf (\ODEstate_t )     -  \barf  ( \barODEstate_t  )    \bigr] }_{R_t: \ \text{almost linear}   }
									  +    \underbrace{   \bigl[    f(\ODEstate_t,\qsaprobe_t)     -  \barf ( \ODEstate_t  )   \bigr] }_{\tilXi_t : \ \text{bounded disturbance}   } 
\end{aligned} 
\]
That is,  under the assumptions of \Cref{t:scaled_error_rep},
\sfb{Can my audience understand?}
\[
R_t = A (\barODEstate_t)     \bigl[  \ODEstate_t  - \barODEstate_t      \bigr]       +   \epsy^1_t
\]
where,  under the Lipschitz condition on $A$,
\[
 \|  \epsy^1_t \| = o( \|  \ODEstate_t  - \barODEstate_t   \| )  = o(a_t \| \scerror_t  \| )  
\]
This completes the proof of \eqref{e:scaledErrorRep},  with $    \Delta_t    =\epsy^1_t /a_t$.

If     $a_t = g/(1+t)$ we obtain $\dlstep_t =  1/(1+t) = g^{-1} a_t$.  
 Equation \eqref{e:scaledErrorRep} thus implies the approximation \eqref{e:scaledErrorRep_cor1},
where the definition of $ \|\Delta_t  \| $ is modified to include the error from  replacing $A (\barODEstate_t)   $ with its limit $ A^\ocp= A (\theta^\ocp) $.

Consider next the ``larger gain''   $a_t = g/(1+t)^\rho$, with $\rho \in (0, 1)$, so that $\dlstep_t =  \rho/(1+t)$.  
The simpler approximation \eqref{e:scaledErrorRep_cor_rho} follows, in which   $ \|\Delta_t  \| $   has two additional terms:  once again, we replace $A (\barODEstate_t)   $ with its limit $ A^\ocp  $, and also use the approximation  $ r_t  =  o(a_t)$.    
\end{proof}

%------------------------------------------------------------------
%  I wouldn't call this the next goal:\bl{Our next goal is to show that $\scerror_t$ is bounded for all $t\geq 0$. 

Recall   the change of variables: $Y_t \eqdef  \scerror_t -  \XiI_t(\ODEstate_t)$
was introduced as a means to remove  the non-vanishing noise $\tilXi_t$  in \eqref{e:scaledErrorRep}.
\Cref{t:YQSA}  establishes a differential equation  for $\bfmY$, similar to the quasi stochastic approximation algorithm \eqref{e:QSAgen}. 

\begin{proposition}
\label[proposition]{t:YQSA}
Under the assumptions of  \Cref{t:Couple_main}, suppose that $ \dlstep_t \le b a_t$ for  a constant $b$,  and all $t\ge 0$.
Then, the vector-valued process $\bfmY$ satisfies the differential equation,
\begin{equation}
\ddt Y_t   =    
    a_t      \left[A^\ocp Y_t   +  \Delta^Y_t   -  \Upupsilon_t      +    A^\ocp  \XiI_t   \right]
+
 \dlstep_t   [ Y_t   +   \XiI_t   ]
\label{e:YQSA}
\end{equation}
where  $ \XiI_t  =  \XiI_t (\theta^\ocp) $,  and
   $ \| \Delta^Y_t \|   = o(1 + \| Y_t\|)$ as $t\to\infty$.   That is, for scalars $\{ \epsy_t^Y\}$,
\[
\| \Delta_t\| \le \epsy_t^Y \{ 1 + \| Y_t \|\}\,,\qquad   t\ge t_0\,
\]
with $\epsy_t^Y\to 0$ as $t\to\infty$.
\end{proposition}

\begin{proof}
Using the chain rule, we have
\[
\begin{aligned}
\ddt \{  \XiI_t(\ODEstate_t)  \} &=  \{  f(\ODEstate_t, \qsaprobe_t) -\barf(\ODEstate_t) \}  + \partial_\theta \XiI_t(\ODEstate_t) \cdot \{ \ddt \ODEstate_t \}
   \\
            & = \tilXi_t  +   \partial_\theta \XiI_t(\ODEstate_t) \cdot \{ a_t f(\ODEstate_t, \qsaprobe_t)    \}
\end{aligned} 
\]
where the second equation follows from the definition $\tilXi_t \eqdef f(\ODEstate_t, \qsaprobe_t) -\barf(\ODEstate_t)   $,
and the dynamics \eqref{e:QSAgen}.
% that determine $ \ddt \ODEstate_t$.
Rearranging terms we obtain  
\begin{equation}
\tilXi_t  
	=
\ddt \{  \XiI_t(\ODEstate_t)  \} -  a_t   \Upupsilon_t(\ODEstate_t)
\label{e:QSAnoiseRep}
\end{equation} 
 where $ \Upupsilon_t(\ODEstate_t)  =  \partial_\theta \XiI_t \, (\ODEstate_t) \cdot  f(\ODEstate_t, \qsaprobe_t) $
 
The following is then obtained on substitution into \eqref{e:scaledErrorRep}:
\[
\ddt Y_t   =    
    a_t      \left[A (\barODEstate_t) Y_t   +  \Delta_t   - \Upupsilon_t(\ODEstate_t)     +    A (\barODEstate_t)  \XiI_t(\ODEstate_t)   \right]
+
 \dlstep_t   [ Y_t   +   \XiI_t(\ODEstate_t)   ]
\]
To go from this ODE to \eqref{e:YQSA} we must bound the error:
\[
\begin{aligned}
\Delta^Y_t  & \eqdef  \Delta^{\text{a}}_t + \Delta^{\text{b}}_t 
\\[.5em]
\textit{where}\quad
 \Delta^{\text{a}}_t  & =
 \Delta_t
 +
\bigl[ A (\barODEstate_t)  -  A^\ocp \bigr] \bigl(  Y_t    +  \XiI_t   \bigr)
\\
 \Delta^{\text{b}}_t  & =
   A (\barODEstate_t) \bigl(   \XiI_t(\ODEstate_t)    -      \XiI_t \bigr)
	 -  \bigl(  \Upupsilon_t(\ODEstate_t)    -  \Upupsilon_t  \bigr)
+
\frac{\dlstep_t }{a_t}      \bigl(   \XiI_t(\ODEstate_t)    -      \XiI_t \bigr)
\end{aligned} 
\]
We have    $ \| \Delta_t \|   = o(1 + \| Y_t\|)$  because of  the prior assertion that $ \|\Delta_t  \| =  o(  \| \scerror_t  \| )  $ as $t\to\infty$,   and the assumption that $\XiI_t(\ODEstate_t)  $ is bounded in $t$  (recall \eqref{e:PoissonA1}).
\Cref{t:Fast_rho} combined with Lipschitz continuity of $A$ then implies that  $ \|   \Delta^{\text{a}}_t  \|   = o(1 + \| Y_t\|)$.

To complete the proof, we must bound the error in replacing $\ODEstate_t$   with $\theta^\ocp$ in each appearance in $ \Delta^{\text{b}}_t $.   
The representation \eqref{e:NoiseInt} combined with (A5) implies that for a Lipschitz constant $\Lip$,
\[
\begin{aligned}
\| A (\barODEstate_t)  -  A^\ocp \|   \le  \Lip \| \barODEstate_t -\theta^\ocp \|      
\qquad 
\| \XiI_t(\ODEstate_t) - \XiI_t  \|    \le \Lip \| \ODEstate_t -\theta^\ocp \|   % \le  a_t \Lip [ \| \scerror_t \|    +o(1)  ]
\end{aligned} 
\]
and hence both error terms are vanishing, and also  $   \| \Upupsilon_t(\ODEstate_t)    -  \Upupsilon_t  \|           = o(1)$  by Lipschitz continuity of $  \partial_\theta \XiI_t(\theta)$:   from \eqref{e:NoiseInt}
and (A5):
\[
  \partial_\theta \XiI_t(\theta)   =   \haA(\theta,\qsaprobe_0) - \haA(\theta,\qsaprobe_t)
\]
These bounds show that $ \| \Delta^{\text{b}}_t   \| = o(1)$.
\end{proof}

%
%\clearpage
%
%\rd{To revise  --- these propositions will go eventually}

%It is relatively easy to establish a convergence rate of $1/t^\rho$ using the gain $a_t = 1/(1+t)^\rho$, with $\rho<1$:

% \begin{proposition}
%\label[proposition]{t:Couple_rho}
%Suppose that (A1)--(A5) hold, and that $A^\ocp$ is Hurwitz.  
%Suppose the  gain is  $a_t = g/(1+t)^\rho$, with  $\rho<1$.  Then the following hold:
%\begin{equation}
% \begin{aligned}
%  \scerror_t   &= \barY  +  \XiI_t   + o(1)  
% \\
%\ODEstate_t & =  \theta^\ocp  +   a_t [\barY  +  \XiI_t ] + o(a_t) 
%\end{aligned} 
%\label{e:QSAcouple_rho<1}
%\end{equation}
%where
%\begin{equation}
% \barY  =
%  A^{-1} \barUpupsilon
%\label{e:barY<1}
%\end{equation}
%When $a_t=1/(1+t)$, the same representations in (\ref{e:QSAcouple_rho<1}, \ref{e:barY<1}) hold for ($\scerror_t, \ODEstate_t$) is $I+gA^*$ is Hurwitz.
%\end{proposition}

\begin{proof}[Proof of \Cref{t:Couple_main}]  

First rewrite \eqref{e:YQSA} as
\[
\ddt Y_t   =    
a_t      \left[(A^\ocp + \frac{r_t}{a_t}I) Y_t   +  \Delta^Y_t   -  \Upupsilon_t      +    (A^\ocp + \frac{r_t}{a_t}I)   \XiI_t   \right]
\]
where $r_t/a_t = o(1)$ for $\rho<1$, and $r_t/a_t \equiv 1$ if $\rho=1$. The above can be regarded as a linear QSA ODE with vanishing disturbance. 
Let $r(\rho) = \ind\{\rho=1\}$  (equal to zero for $\rho<1$, and $r(1) =1$).  
Under the condition that $A^\ocp + r(\rho)I$ is Hurwitz, the proof of \Cref{t:convergenceV_QSA} can be used with no significant changes to establish convergence:

%\rd{I started to clarify, and then see the argument is not correct.   The result below is valid only for $\rho<1$  ($r_t$ has been dropped).   Need to consider the two cases separately.  Below is fine for $\rho<1$, and then need to say what happens for $\rho=1$ since $r_t$ cannot be ignored. 
%\\
%Please think hard, and see if you can find a more elegant proof.  It may not be possible.  That is, we may have to go with "case 1 and case 2"  }
\[
\barY = \lim_{t\to\infty}  Y_t  =  \lim_{T\to \infty}\frac{1}{T} \int_0^T  \bigl[\XiI_t   + [A^\ocp + r(\rho)I]^{-1}   \partial_\theta \XiI_t  \cdot  f(\theta^\ocp, \qsaprobe_t) \bigr]\, dt
=
  [A^\ocp + r(\rho)I]^{-1} \barUpupsilon
\]
where the second equality holds because $\displaystyle \lim_{T\to \infty}\frac{1}{T} \int_0^T  \XiI_t \, dt  =0$ under \eqref{e:NoiseInt}, and from the definition of $\barUpupsilon$ in (A5).  This gives the coupling result  $  \scerror_t   = \barY  +  \XiI_t   + o(1) $.

The second approximation  in \eqref{e:QSAcouple_rho} follows from the first:   applying the definition \eqref{e:scaled_error}
gives
\[
\ODEstate_t   =  \barODEstate_t  +   a_t [\barY  +  \XiI_t ]  + o(a_t)
\]

For $\rho < 1$ we have   $ \barODEstate_t  = \theta^\ocp  +  o(a_t)$ since $\barODEstate_t$ converges to $\theta^\ocp$ faster than $t^{-N}$ for any $N$.  
%\notes{Error!  Error!!  exponentially fast} 

 For $\rho = 1$, \Cref{t:Fast_rho} (i) implies that $ \barODEstate_t  = \theta^\ocp  +  O(t^{-\varrho_0})$ where $\text{Real}(\lambda) < -\varrho_0$ for every eigenvalue $\lambda$ for $A^\ocp$. Therefore, $\barODEstate_t  = \theta^\ocp  +  o(t^{-1})$ if  $I+A^\ocp$ is Hurwitz.
\end{proof}

%\newlength{\dhatheight}
%\newcommand{\doublehat}[1]{%
%    \settoheight{\dhatheight}{\ensuremath{\hat{#1}}}%
%    \addtolength{\dhatheight}{-0.25ex}%
%    \hat{\vphantom{\rule{1pt}{\dhatheight}}%
%    \smash{\hat{#1}}}}

% Similar arguments establish the optimal rate of convergence, but subject to stronger assumptions:
%
% \begin{proposition}
%\label[proposition]{t:Couple1}
%Suppose that (A1)--(A5) hold, the gain is $a_t = 1/(1+t)$,  
%and that $I+A^\ocp$ is Hurwitz.   Then the convergence rate is $1/t$:
%\begin{equation}
% \begin{aligned}
%  \scerror_t   &= \barY  +  \XiI_t   + o(1)  
% \\
%\ODEstate_t & =  \theta^\ocp  +   a_t [\barY  +  \XiI_t ] + o(a_t) 
%\end{aligned} 
%%\label{e:QSAcouple1}
%\end{equation}
%\end{proposition}

\begin{proof}[Proof of \Cref{t:Couple_rhoRP}]
Let $\ODEstate^{\text{RP}}_T$ be defined by \eqref{e:QSAgenRPc}.
By \Cref{t:Fast_rho} (ii), the convergence rate of $\barODEstate_t$ is exponential in $t$ and we obtain the approximation 
\begin{equation}
\ODEstate^{\text{RP}}_T  - \theta^\ocp
\eqdef
\frac{1}{T-T_0}  \int_{T_0}^T [\ODEstate_t - \theta^\ocp] \, dt  =    \frac{1}{T-T_0}  \int_{T_0}^T [\ODEstate_t  -  \barODEstate_t] \, dt + o(1)
\label{e:QSA-R-P}
\end{equation} 
Combining the definition $    a_t             \scerror_t    =  \ODEstate_t  -  \barODEstate_t$ and  
\[
 \ddt \scerror_t   =     a_t      A^\ocp       \scerror_t +  a_t \Delta_t    +       \tilXi_t   
\]
(see \eqref{e:scaledErrorRep_cor_rho}),
gives by the
Fundamental Theorem of Calculus:
\[
   \scerror_T  -    \scerror_{T_0}     =                A^\ocp   \int_{T_0}^T [ \ODEstate_t  -  \barODEstate_t] \, dt 
   		+   \int_{T_0}^T    [ a_t \Delta_t    +       \tilXi_t   ] \, dt
\]
Under (A4),  we have $\|a_t \Delta_t  \| = a_t O(\|  \ODEstate_t  -  \barODEstate_t\|)$.
%\notes{SH: $a_t\Delta_t = o(\|  \ODEstate_t  -  \barODEstate_t\|)$ from page 31 $R_t$. 
%The residual in Taylor series expansion of $R_t$ is $O(\|\ODEstate_t  -  \barODEstate_t\|^2)$ because we are assuming $A$ is Lipschitz continuous? 
%\\
%\bl{You tell me!  Do you see a bug?}
%}
 It then follows from \Cref{t:Couple_main} (i) that $\displaystyle \int_{T_0}^T a_t \Delta_t \, dt = o(1)$ for $\rho \in (\half, 1)$. 
 Therefore,
\begin{equation}
 \int_{T_0}^T [ \ODEstate_t  -  \barODEstate_t]  \, dt = [A^\ocp]^{-1}  \Bigl\{   \scerror_T  -    \scerror_{T_0}    -             
  \int_{T_0}^T           \tilXi_t  \, dt  + o(1)  \Bigr\} 
  \label{e:tilTheta-Z-Xil}
\end{equation}
Next recall \eqref{e:QSAnoiseRep} and \eqref{e:NoiseInt}, which gives 
\[
\begin{aligned} 
  \int_{T_0}^T      \tilXi_t   \, dt    
     &  =    [  \XiI_T (\ODEstate_T ) -    \XiI_{T_0} (\ODEstate_{T_0}  ) ] 
-     \int_{T_0}^T     
 a_t  \Upupsilon_t(\ODEstate_t)     \, dt      
 \\
   &  =   [  \XiI_T (\ODEstate_T ) -    \XiI_{T_0} (\ODEstate_{T_0}  ) ] 
  	 -       \int_{T_0}^T     
 a_t \Upupsilon_t    \, dt      +  
    \int_{T_0}^T     
 a_t O(\|  \ODEstate_t  -  \barODEstate_t\|)   \, dt  
 \end{aligned} 
\]
where $\displaystyle \int_{T_0}^T     
a_t O(\|  \ODEstate_t  -  \barODEstate_t\|)   \, dt  = o(1)$. 
Recalling the definition $Y_t \eqdef  \scerror_t -  \XiI_t(\ODEstate_t)$ in \eqref{e:Y} gives, 
%\notes{The big-oh term is $o(  \int_{T_0}^T     
%   \| \tilODEstate_t \|   \, dt  )$.  Think it over!   The proof is trivial if the latter integral tends to infinity. See if you can find a cleaner way to eliminate this term.}

\begin{equation}
\scerror_T  -    \scerror_{T_0}    -             
\int_{T_0}^T           \tilXi_t  \, dt  
=  Y_T  -  Y_{T_0}    +    \int_{T_0}^T     a_t \Upupsilon_t    \, dt + o(1)
\label{e:tilTheta-Y-Xil}
\end{equation}
Combining \eqref{e:tilTheta-Z-Xil} and \eqref{e:tilTheta-Y-Xil} gives
\[
  \frac{1}{T-T_0} \int_{T_0}^T[ \ODEstate_t  -  \barODEstate_t]  \, dt =  \frac{1}{T-T_0}  [A^\ocp]^{-1}  \Bigl\{    Y_T  -  Y_{T_0}    +    \int_{T_0}^T     a_t \Upupsilon_t    \, dt + o(1)   \Bigr\} 
\]
The last integral is the crucial term.  Write  
  \[
   \tilUpupsilon_t =  \Upupsilon_t - \barUpupsilon
   \quad \textit{and} \quad
  \tilUpupsilon_t^I =  \int_0^t   [\Upupsilon_r  - \barUpupsilon] \, dr
 \]
  where $\barUpupsilon$ is the ergodic mean introduced in (A5), and both terms are  bounded in $t$ by assumption.  We then obtain
 \[
    \int_{T_0}^T     a_t \Upupsilon_t    \, dt      =   \barUpupsilon  \int_{T_0}^T     a_t     \, dt 
    	+     \int_{T_0}^T     a_t  \, d \tilUpupsilon_t^I
\]
  Then, using integration by parts, the second term above is bounded:
 \[
 \begin{aligned} 
  \int_{T_0}^T     a_t  \, d \tilUpupsilon_t^I      
     	&=    a_t    \tilUpupsilon_t^I \Big|_{t=T_0}^T   -    \int_{T_0}^T  [ \ddt a_t  ]   \tilUpupsilon_t^I\, dt
     \\
	&=    [ a_T    \tilUpupsilon_T^I  -a_{T_0}     \tilUpupsilon_{T_0} ^I]     + \rho    \int_{T_0}^T \frac{1}{(1+t)^{1+\rho}}   \tilUpupsilon_t^I\, dt   
\end{aligned} 
 \]
 This establishes the conclusion:   if $\barUpupsilon$ is non-zero, then averaging* does not achieve $1/T$ convergence rate,   since $  \int_{T_0}^T     a_t     \, dt 
$ is $O(T^{1-\rho})$. And, if $\barUpupsilon$ is zero, the $1/T$ convergence rate holds for $\ODEstate^{\text{RP}}_T  - \theta^\ocp$.   
% \text{\bl{\tt (happy face emoji) }}
  \end{proof}

\end{document}